\documentclass{amsart}
\usepackage{amssymb}
\usepackage[dvips]{epsfig}
\usepackage{graphicx}
\usepackage{color}

 \usepackage[latin1]{inputenc}
 \usepackage[T1]{fontenc}
 \usepackage[normalem]{ulem}
\usepackage{verbatim}
 \usepackage{graphicx}
\usepackage[latin1]{inputenc}
\usepackage{amsmath}
\usepackage{amssymb}
\usepackage{amsfonts}
\usepackage{amsthm}
\usepackage{bbm}

\newcommand{\R}{\mathbb R}
\newcommand{\p}{\partial}

\newcommand{\g}{\gamma}

\renewcommand{\l}{\lambda}
\renewcommand{\lg}{\Lambda}

\newcommand{\sgn}{{\rm sgn}}
\newcommand{\N}{\mathbb{N}}
\newtheorem{theorem}{Theorem}[section]
\newtheorem{lemma}[theorem]{Lemma}
\newtheorem{proposition}[theorem]{Proposition}
\newtheorem{corollary}[theorem]{Corollary}
\theoremstyle{remark}
\newtheorem{remark}{Remark}[section]
\theoremstyle{definition}
\newtheorem{definition}[theorem]{Definition}

\newtheorem*{merci}{Acknowledgements}
\numberwithin{equation}{section}
\begin{document}

\title[The Cauchy problem for fKP equations]{The Cauchy problem for the fractionaL Kadomtsev-Petviashvili equations}
\author[F. Linares]{Felipe Linares}
\address{ IMPA\\ Estrada Dona Castorina 110\\ Rio de Janeiro 22460-320, RJ Brasil}
\email{ linares@impa.br}
\author[D. Pilod]{Didier Pilod}
\address{Instituto de Matem\' atica, Universidade Federal do Rio de Janeiro, Caixa Postal 68530 CEP 21941-97, Rio de Janeiro, RJ Brasil}
\email{didier@im.ufrj.br}
\author[J.-C. Saut]{Jean-Claude Saut}
\address{Laboratoire de Math\' ematiques, UMR 8628,\\
Universit\' e Paris-Sud et CNRS,\\ 91405 Orsay, France}
\email{jean-claude.saut@u-psud.fr}

\date{May 22, 2017}
\maketitle

\begin{abstract}
The aim of this paper is to prove various ill-posedness and well-posedness results on the Cauchy problem associated to a class of  fractional Kadomtsev-Petviashvili (KP) equations including the KP version of the Ben\-ja\-min-Ono and Intermediate Long Wave equations.
\end{abstract}

\large
\section{Introduction}

We continue here our study of \lq\lq weak\rq\rq\, dispersive perturbations of the Burgers equation. We will consider here {\it fractional} KP type equations

\begin{equation} \label{fKP}
\left\{\begin{array}{l}
u_t+uu_x-D_x^\alpha u_x+\kappa\partial_x^{-1} u_{yy}=0,\quad (x,y,t) \in \R^2\times \R_+,\\
u(x,y,0)=u_0(x,y),\quad (x,y) \in \mathbb R^2 ,
\end{array}\right.
\end{equation}
where $0<\alpha \le 2$, $\kappa =1$ corresponds to the fKP-II equation and $\kappa =-1$ to the fKP-I equation. Here $D^{\alpha}_x$ denotes the Riesz potential of order $-\alpha$ in the $x$ direction, \textit{i.e.} $D^{\alpha}_x$ is defined via Fourier transform by 
$$\big(D^{\alpha}_xf\big)^{\wedge}(\xi,\eta)=|\xi|^{\alpha}\widehat{f}(\xi,\eta) \, .$$

This equation can be thought as a two-dimensional weakly transverse version  of the fractional KdV equation (fKdV)
\begin{equation}\label{fKdV}
u_t+uu_x\pm D_x^\alpha u_x=0 .
\end{equation}

Actually the (formal) derivation in \cite{KaPe}  of the Kadomtsev-Petviashvili equation is independent of the dispersion and $x$ and concerns only the transport part of the KdV equation 
\begin{equation}\label{KdV}
u_t+u_x+uu_x+(\frac{1}{3}-T)u_{xxx}=0 ,
\end{equation}
where $T\geq 0$ is the surface tension parameter.

More precisely, it consists in looking for a   weakly transverse perturbation 
of the one-dimensional transport equation
\begin{equation}\label{transp}
u_t+u_x=0.
\end{equation}
This perturbation is obtained by a Taylor expansion of the dispersion relation  $\omega(k_1,k_2)=\sqrt{k_1^2+k_2^2}$ of the two-dimensional linear wave equation assuming that $|k_1|\ll 1 $ and $\frac{|k_2|}{|k_1|}\ll 1.$ 
Namely, one writes formally
$$\omega(k_1,k_2)\sim\pm k_1\left(1+\frac12\frac{k_2^2}{k_1^2}\right)$$
which, with the $+$ sign say, amounts, coming back to the physical variables,  to adding a nonlocal term to the transport equation,  
\begin{equation}\label {perttransp}
u_t+u_x+\frac{1}{2}\partial_x^{-1}u_{yy}=0.
\end{equation}
Here the operator $\partial_x^{-1}$ is defined via Fourier transform,
$$\widehat{\partial_x^{-1}f}(\xi,\eta)=\frac{1}{i\xi}\widehat{f}(\xi,\eta) \, .$$

The same formal procedure is applied in \cite{KaPe} to the KdV equation \eqref{KdV}, 
assuming that the transverse dispersive effects are of the same order as the x-dispersive and nonlinear terms,   yielding the KP equation in the form

\begin{equation}\label{KPbrut}
u_t+u_x+uu_x+(\frac{1}{3}-T)u_{xxx}+\frac{1}{2}\partial_x^{-1}u_{yy}=0.
\end{equation}
The KP-II equation is obtained when $T<\frac{1}{3},$ the KP-I when $T>\frac{1}{3}.$\footnote{Of course this formal argument has to be justified in \lq\lq concrete\rq\rq \, situations. See for instance  \cite{La} and the references therein for the justification in the context of water waves.}

It is thus quite natural to apply this formal process to \eqref {fKdV} to obtain \eqref {fKP}. The sign of the x-dispersive term in \eqref{fKdV} will determine that of $\partial_x^{-1}u_{yy}$ in \eqref {fKP}.

Note also that \eqref{fKP} with $\alpha=1$ is the relevant KP version of the Benjamin-Ono equation (KP-BO). Very similar to the KP-BO equation is the KP version of the intermediate long wave equation (ILW):
\begin{equation}\label{ILW}
u_t+uu_x-\mathcal L_{ILW}u_x=0,
\end{equation}
where $\widehat{\mathcal L_{ILW} f}(\xi)=p_{ILW}(\xi)\hat f (\xi),$ \quad $p_{ILW}(\xi)=\xi\coth (\delta\xi)-\frac{1}{\delta}, \delta >0$. The ILW equation is a model for long, weakly nonlinear internal waves, $\delta$ being proportional to the depth of the bottom layer (see \cite{BLS} for a rigorous derivation and study of the ILW and related equations).
The Benjamin--Ono equation is obtained in the infinite depth limit, $\delta \to +\infty.$

Contrary to their one-dimensional version, the  KP-ILW equation 
\begin{equation}\label{KP-ILW}
u_t+uu_x-\mathcal L_{ILW}u_x\pm \partial_x^{-1}u_{yy}=0,
\end{equation}
or the KP-BO equation 
\begin{equation}\label{KP-BO}
u_t+uu_x-\mathcal Hu_{xx}\pm \partial_x^{-1}u_{yy}=0,
\end{equation}
where $ \mathcal H$ is the Hilbert transform, are not known to be completely integrable.

It is worth noticing that in the context of internal waves, only the $+$ sign in \eqref{KP-BO}, \eqref{KP-ILW} is relevant. In fact, when surface tension is taken into account in this setting, the correction to the one-directional BO or ILW equations is to adding a new dispersive term $-\beta u_{xxx}$ (see \cite{Ben}) so that the corresponding KP version writes

$$u_t+uu_x-\mathcal Hu_{xx}+\beta u_{xxx}+ \partial_x^{-1}u_{yy}=0,$$

or

$$
u_t+uu_x-\mathcal L_{ILW}u_x+\beta u_{xxx}+ \partial_x^{-1}u_{yy}=0,
$$
that are variant of the classical KP equations which we will not consider in the present paper (see for instance \cite{CGH} for the solitary wave solutions).

\vspace{0.5cm}
As in our previous works \cite{LPS2, LPS3}, in order to study the influence of a \lq\lq weak\rq\rq \, dispersive perturbation on a quasilinear hyperbolic equation we  
have chosen to fix the quadratic nonlinearity and to vary the strength of the dispersion.

Equation \eqref{fKP} can thus be also seen  as  a toy model to investigate the effects of a \lq\lq KP like\rq\rq \, perturbation on the Burgers equation. When $\alpha =\pm1/2$ it has some 
links with the {\it full dispersion}  KP equation derived in \cite{La} and considered in \cite{LS} as an alternative model to  KP (with less unphysical shortcomings) for  gravity-capillary surface waves in the weakly transverse regime:

 \begin{equation}\label{FDbis}
\partial_t u+\tilde c_{WW}(\sqrt\mu|D^{\mu}|)(1+\mu \frac{D_2^2}{D^2_1})^{1/2} u_x+\mu \frac{3}{2} uu_x=0,
\end{equation}
with 
$$
\tilde c_{WW}(\sqrt \mu k)=(1+\beta \mu k^2)^{\frac{1}{2}}\left(\frac{\tanh \sqrt \mu k}{\sqrt \mu k}\right)^{1/2},
$$
where $ \beta\geq0$ is a dimensionless coefficient measuring the surface tension
effects and
$$|D^\mu|=\sqrt{D_1^2+\mu D_2^2}, \quad D_1=\frac{1}{i} \partial_x,\quad D_2=\frac{1}{i}\partial_y.$$
For the sake of completeness we recall here the regime leading to the asymptotic model \eqref{FDbis} for surface water waves.

Denoting by $h$ a typical depth of the fluid layer, $a$ a typical amplitude of the wave, $\lambda_x$ and $\lambda_y$ typical wave lengths in $x$ and $y$ respectively, the relevant regime here is when

$$\mu\sim \frac{a}{h}\sim \left(\frac{\lambda_x}{\lambda_y}\right)^2\sim \left(\frac{h}{\lambda_x}\right)^2\ll 1.$$
For purely gravity waves, $\beta =0, $ \eqref {FDbis} becomes
\begin{equation}\label{FDter}
\partial_t u+ c_{WW}(\sqrt\mu|D^{\mu}|)(1+\mu \frac{D_2^2}{D^2_1})^{1/2} u_x+\mu \frac{3}{2} uu_x=0,
\end{equation}
with 
$$
 c_{WW}(\sqrt \mu k)=\left(\frac{\tanh \sqrt \mu k}{\sqrt \mu k}\right)^{1/2},
$$
 which can be viewed as the KP version of the (one-dimensional) Whitham equation
\begin{equation}\label{Whit}
\partial_t u+ c_{WW}(\sqrt\mu|D^{\mu}|) u_x+\mu \frac{3}{2} uu_x=0.
\end{equation}
One recovers formally the classical KP-II ($\beta<1/3$) and KP-I ($\beta>1/3$) equations from \eqref{FDbis}, by keeping the first order term in the expansions with respect to $\sqrt\mu k$ of the nonlocal operators appearing in the equations.

Observe also that \eqref{FDbis},  behave for high frequencies as \eqref{fKP} with  $\alpha=-\frac{1}{2}$ when $\beta=0$ and  $\alpha =\frac{1}{2}$ when $\beta>0.$ Thus, similarly to the Whitham equation, \eqref{FDbis} displays various regimes depending on the frequency range. In particular it is proven in \cite{EhGr} that \eqref{FDbis} with $\beta >0$ possesses solitary wave solutions close to the \lq\lq lumps\rq\rq \, of the KP-I equation. We plan to come back to those issues in a next paper.

\vspace{0.3cm}
In addition to the $L^2$ norm, \eqref{fKP} conserves formally the energy (Hamiltonian)

\begin{equation}\label{HamfKP}
H_\alpha(u)=\int_{\R^2} (\frac{1}{2}|D_x^{\frac\alpha2} u|^2-\kappa\frac{1}{2}|\partial_x^{-1}u_y|^2-\frac{1}{6} u^3).
\end{equation}
The corresponding energy space is 
$$X^{\frac{\alpha}2}= \lbrace u\in L^2(\R^2) \ :  \ D^{\frac\alpha2}_x u, \ \partial_x^{-1}u_y\in L^2(\R^2)\rbrace.$$

The first question for \eqref{fKP} is to which values of $\alpha$ correspond to the $L^2$ and the energy critical cases? 

For the generalized KP equations 
\begin{equation}\label{gKPI}
u_t+u^pu_x+u_{xxx}+\kappa \partial_x^{-1}u_{yy}=0,
 \end{equation} 
 the corresponding  values of $p$ are respectively $p=4/3$ and $p=4$ (see \cite{deBS}). 
 
 One checks readily that the transformation
 $$u_\lambda(x,y,t)=\lambda^\alpha u(\lambda x, \lambda ^{\frac{\alpha+2}{2}} y, \lambda^{\alpha+1} t)$$ leaves \eqref{fKP} invariant. 
 Moreover, $\|u_\lambda\|_{L^2}=\lambda^{\frac{3\alpha -4}{4}}\|u\|_{L^2}$, so that $\alpha=\frac{4}{3}$ is the $L^2$ critical exponent. Note that the BO-KP and the ILW-KP equations are $L^2$  supercritical.

\vspace{0.3cm}
The following fractionary Gagliardo-Nirenberg  inequality is actually a special case of    Lemma 2.1  in \cite{BLT} which considers only  $1\leq \alpha\leq 2$ but a close inspection at the proof reveals that it is still valid when $\alpha \geq\frac{4}{5}.$

\begin{lemma}\label{BLT}
Let $\frac{4}{5}\leq\alpha < 1.$ For any $f\in X^{\frac{\alpha}2}$ one has 
$$\|f\|_{L^3}^3\leq c\|f\|_{L^2}^{\frac{5\alpha-4}{\alpha+2}}\|f\|_{H^{\frac\alpha2}_x}^{\frac{18-5\alpha}{2(\alpha+2)}}\|\partial_x^{-1}f_y\|_{L^2}^{\frac12} \, ,$$
where $\|\cdot\|_{H^{\frac{\alpha}2}_x}$ denotes the natural norm on the space $$H_x^{\frac\alpha2}(\R^2)=\lbrace f\in L^2(\R^2) \ : \  D^{\frac\alpha2}_xf \in L^2(\R^2)\rbrace.$$
\end{lemma}
Lemma \ref{BLT} implies obviously the embedding $X^{\frac{\alpha}2} \hookrightarrow L^3(\R^2)$ if $\alpha \geq\frac{4}{5}.$
 
The energy critical value $\frac{4}{5}$ of $\alpha$ is obviously related to the non existence of localized solitary waves. Actually, one has by Pohozaev type arguments  \cite{LPS3}:
\begin{proposition}\label{nonex}
Assume that $0<\alpha\leq \frac{4}{5} $ when $\kappa =-1$ or that $\alpha$ is arbitrary when $\kappa =1.$ Then \eqref{fKP} does not possess non trivial  solitary waves in the space $X^{\frac{\alpha}2} \cap L^3(\R^2).$
\end{proposition}

\vspace{0.3cm}
The paper is organized as follows. In a first section we review some general facts on the linear and nonlinear Cauchy problem for equations such as \eqref{fKP}. They concern properties of the linear group, {\it eg} a local smoothing effect for the fKP-II group, the so-called constraint problem for the linear groups (and its extension to the nonlinear problem)    and various results on the nonlinear Cauchy problem, including the existence of global weak solutions via the conservation of energy in the fKP-I case. 

The next sections are the core of our work. We first prove  that the Cauchy problem for the fKP-I equation and for the fKP-II equation when $\alpha <\frac{4}{3}$ cannot be solved by a Picard iterative scheme based on the Duhamel formula, extending the result in \cite{MST} for the usual KP-I equation. Thus  the fKP-I equation when $\alpha \leq 2$ and  the fKP-II equation when $\alpha <\frac{4}{3}$ are {\it quasilinear} while the fKP-II equation when $\alpha>\frac{4}{3}$ is {\it semilinear}. This illustrates the subtility of the distinction between \lq\lq semilinear\rq\rq \, and \lq\lq quasilinear\rq\rq \, in the context of KP type equations since it has been proven in \cite{ST} that the fifth order ($\alpha =4$) KP-I equation is semilinear, in particular  the Cauchy problem can be solved by a Bourgain method as in \cite{Bou} or \cite{Ta-Tz}.

 Then in Sections 4 and 5,  we state and prove our main existence result, namely that for any $\alpha >0$ one can solve the local Cauchy problem for \eqref{fKP} for initial data in a space strictly larger than the \lq\lq hyperbolic\rq\rq \, space $H^{2^+}(\R^2),$ actually in $X^s(\R^2),\; s>s_\alpha=2-\frac{\alpha}{4}.$  This is the first general result of this type for this kind of equations, in particular for the BO-KP equation ($\alpha =1$). It also applies to nonhomogeneous symbols in particular  to the ILW-KP equation or to the KP version of the Whitham equation with surface tension.

Note however that Hadac (\cite{Ha}) obtained the local well-posedness of the fKP-II equation in the $L^2-$ subcritical case $\alpha>\frac{4}{3}$ for initial data in the anisotropic Sobolev space $H^{s_1, s_2}(\R^2),\;s_1>\max (1-\frac{3}{4}\alpha,\frac{1}{4}-\frac{3}{8}\alpha),\;s_2\geq 0.$

The main ingredient in the well-posedness analysis is a refined Strichartz estimate in the same spirit of \cite{K} (see also \cite{KoTz3}).

\vspace{0.5cm}
\noindent{\bf Notations.} We will denote $\|\cdot\|_p$ the norm in the Lebesgue space $L^p(\R^2),\; 1\leq p\leq \infty$ and $\|\cdot\|_{H^s}$ the norm in the Sobolev space $H^s(\R^2),\; s\in \R.$ We will denote $\widehat {f}$ or $\mathcal F(f)$ the Fourier transform of a tempered distribution $f.$ For any $s\in \R,$ we define $D_x^s f$ and $D_y^sf$ by  Fourier transform 
$$\widehat{D_x^s f}(\xi,\eta)=|\xi|^s \hat{f}(\xi,\eta) \quad \text{and} \quad \widehat{D_y^s f}(\xi,\eta)=|\eta|^s \hat{f}(\xi,\eta) \, .$$ 

Finally, we set  $$H^s_{-1}(\R^2)=\lbrace f\in H^s(\R^2) \ : \ {\mathcal F}^{-1}\big(\frac{\widehat{f}(\xi,\eta)}{\xi}\big)\in H^s(\R^2)\rbrace.$$

For any $s \ge 0$, we define the space $X^s(\mathbb R^2)$, which is well-adapted to the equation  \eqref{fKP}, by the norm 
\begin{equation} \label{Xnorm}
\| f\|_{X^s} := \left(\|J^s_x f\|_{L^2(\R^2)}^2+\|\partial^{-1}_x\partial_y f\|_{L^2(\R^2)}^2\right)^{\frac12}
\end{equation}
and
$$
X^{\infty}(\R^2)=\underset{s\ge0}{\cap}X^s(\R^2).
$$

Finally, we define the space $W_x^{1,\infty}(\R^2)$, by the norm
\begin{equation}\label{w1infty}
\|f\|_{W^{1,+\infty}_x}=\|f\|_{L^{\infty}_{xy}}+\|\partial_xf\|_{L^{\infty}_{xy}}.
\end{equation}

\section{Basic results on the Cauchy problem}
\subsection{Basic facts on the linear problem}

The linear part in \eqref{fKP} defines, for any $\alpha>-1$,  a unitary group $S_\alpha(t)$ in $L^2$ and all $H^s(\R^2)$  Sobolev spaces, unitarily equivalent via Fourier transform to the Fourier multiplier
$$
e^{it(|\xi|^\alpha\xi-\kappa \frac{\eta^2}{\xi})}.
$$
On the other hand, though the solutions are continuous in time they will not be in general differentiable in time, even for smooth (say Schwartz class) initial data. Actually a necessary and sufficient condition on the initial data $u_0$ for $u_t$ to be a bounded function of $t$ with values in $L^2(\R^2)$ is that 
$$\xi^{-1}\eta \widehat{u_0}\in L^2(\R^2).$$
This fact is linked to the so-called \lq\lq constraint problem\rq\rq \,  (see subsection 2.4 below) and was already observed for the "linear diffractive pulse equation"
$$2u_{tx}=\Delta_y u$$
in \cite{AR}.

Note also that the linear part of the full dispersion KP equation \eqref{FDbis} does not suffer from this shortcoming (see \cite{LS}) and actually 
this was one of the reasons to introduce this kind of equations.
\subsection{An elementary result}
Viewing \eqref{fKP} as a skew-adjoint perturbation of the Burgers equation, one easily establishes the elementary local well-posedness result (see \cite {IN, S} and \cite{MST} for a simpler proof):
\begin{theorem}\label{triv}
Let $u_0\in H^s_{-1}(\R^2), s>2.$ Then there exist  $T=T(\|u_0\|_{H^s})>0$ and a unique solution  $u\in C([0,T];H_{-1}^{s}(\R^2)), \, u_t \in C([0,T];H^{s-3}(\R^2))$. \\
Furthermore, the map $\; u_0 \mapsto u $ is continuous from 
 $\;H^s_{-1}(\R^2)$ to\newline $\; C([0,T];H_{-1}^{s}(\R^2)) $. Moreover $\|u(\cdot,t)\|_{L^2}$ and $H_\alpha(u(\cdot,t))$ are conserved on $[0,T].$
\end{theorem}

\begin{proof}
The proof of local well-posedness is obtained  by a standard compactness method followed by the Bona-Smith trick for the continuity properties.
To justify rigorously the conservation of the Hamiltonian one use as in \cite{Mol} an exterior regularization of the 
KP equation by a sequence of smooth functions $ u_\epsilon $ that cut the 
low frequencies. Namely, one introduces for $ \epsilon>0 $ the 
function $ u_\epsilon $ defined by
\begin{equation}
\widehat{u}_{\epsilon}(\xi,\eta):=
\left\{
\begin{array}{ll}
1 & \hbox{ if }
\epsilon < |\xi|<\frac{1}{\epsilon}\hbox{ and }
 \epsilon < |\eta |<\frac{1}{\epsilon}, \\
0 & \hbox{otherwise.}
\end{array}
\right.
\label{b2}
\end{equation}
Obviously $ u_\epsilon\in H^{s_1,s_2}_{-k} $ for any 
$ (s_1,s_2)\in \R_{+}\times\R_{+},\; k\in \N $.
Furthermore,
\begin{equation}
\label{b4}
\lim_{\epsilon\to 0} \|u_\epsilon\ast v-v\|_{L^p}=0,  \quad 2\le p\le \infty,
\quad \forall v\in H^{\infty}(\R^2).
\end{equation}

\end{proof}
\subsection{Local smoothing for fKP-II equations}

It is well-known (\cite {S}) that the  linear   KP-II equation ($\alpha=2, \kappa=1$) displays a local smoothing property, also shared by smooth solutions of the nonlinear problem (\cite{S}, Theorem 3.1). On the other hand, Ginibre and Velo \cite{GV2} have proven a local smoothing property for the fractional Korteweg-de Vries equation \eqref{fKdV}
when $\alpha>\frac{1}{2}$ leading to the global existence of weak $L^2$ or finite energy solutions. Combining the two approaches we can prove a local smoothing property for the linear fKP-II equation

\begin{equation}\label{lfKP-II}
u_t-D_x^\alpha u_x+\partial_x^{-1}u_{yy}=0.
\end{equation}
One gets (compare with Proposition 2.6 in \cite{S}):
\begin{proposition}
Let $\alpha >\frac{1}{2}$ and $s\geq 0, u_0\in H_{-1}^s(\R^2).$ Then the solution $u$ of \eqref{lfKP-II} satisfies for any $R>0$ and $T>0$ 
$$D_x^{\frac{\alpha}2}D_x^{s_1}D_y^{s_2}u,\quad D_y^s \partial_x^{-1}u_y\in L^2((-T,T)\times (-R,R)\times \R),\quad s_1+s_2=s.$$
\end{proposition}
\begin{proof}
It will be convenient to write \eqref{lfKP-II} as

\begin{equation}
    \label{lfKPbis}
    \left\lbrace
   \begin{array}{lcl}
   	 u_t -D_x^\alpha u_x+v_y=0,\\
	 u_y=v_x.
    \end{array}\right.
\end{equation}

We will consider only  $s=0,$ the general case follows by applying the procedure below to $D_1^{s_1}D_2^{s_2}u.$ The following computations can be justified by smoothing the initial data and passing to the limit.

Let $p\in \mathcal C^ {\infty}(\R_x,\R)$ with $p'\geq 0$ and $p'$ compactly supported. We multiply $\eqref{lfKPbis}_1$ by $pu,$  integrate over $\R^2$ to get after several integrations by parts in the last term

\begin{equation}\label{triv1}
\frac{1}{2}\frac{d}{dt}\int_{\R^2} pu^2 dx dy-\int_{\R^2} (D_x^\alpha u_x) p u \,dx dy+\frac{1}{2}\int_{\R^2} p_xv^2 dx dy=0.
\end{equation}

In order to deal with the second term in \eqref{triv1} we use a commutator lemma which is a consequence of Proposition 2.1 in \cite{GV}.

\begin{lemma}\label{comm}\cite{GV}

For any $0<\alpha<1,$

$$[-\partial_x D^\alpha_x,p]=(1+\alpha) D^{\frac{\alpha}2}_x p'D_x^{\frac{\alpha}2} +R_\alpha(p)$$

where $R_\alpha(p)$ is bounded in $L^2(\R).$
\end{lemma}

Since $p$ does not depend on $y$,  $R_\alpha(p)$ is also  bounded in $L^2(\R^2)$ and we deduce from
 Lemma \ref{comm}

\begin{equation}\label{triv2}
\frac{1}{2}\frac{d}{dt}\int_{\R^2} pu^2 \,dx dy+\int_{\R^2}p_x |D_x^{\frac{\alpha}2} u|^2  dx dy+\frac{1}{2}\int_{\R^2} p_xv^2 dx dy\leq C\int_{\R^2}u^2\,dxdy.
\end{equation}
\end{proof}
\subsection{The constraint problem}
The singular term $\partial_x^{-1}u_{yy}$ in KP-like equations induces a (zero mass) constraint
on the solution that has been studied in detail in \cite{MST1}. For the sake of completeness we recall
here the results in \cite{MST1} related to the fKP equations \eqref{fKP}.
We consider first the linear equation.

\begin{equation}\label{lfKP}
u_t-D^\alpha_x u_x+\kappa\partial_x^{-1}u_{yy}=0
\end{equation}

 We denote by $G_\alpha$\footnote{The value of $ \kappa$ is irrelevant here and will take $\kappa =1$ throughout the proof.} the fundamental solution
\[
G_\alpha(x,y,t)={\mathcal F}^{-1}_{(\xi,\eta)\rightarrow (x,y)}
\bigl[
e^{it(\xi|\xi|^{\alpha}-\kappa\eta^2/\xi)}
\bigr].
\]
A~priori, we have only that $G_\alpha(\cdot,\cdot,t)\in {\mathcal S}'(\R^2)$.
Actually, for $t\neq 0$, $G_\alpha(\cdot,\cdot,t)$ has a very particular form described in the next theorem which is essentially contained in  \cite{MST1}.

\begin{theorem}\label{thm2.1} 
Suppose that $\alpha>0$ in\/ {\rm(\ref{lfKP})}.

Then  for $t\neq 0$, there exists 

$$A_\alpha(\cdot,\cdot,t)\in C(\R^2)\cap L^{\infty}(\R^2)\cap C^{1}_{x}(\R^2)$$

($C^{1}_{x}(\R^2)$ denotes the space of continuous functions on\/ $ \R^2 $ which have
 a continuous derivative with respect to the first variable) such that

$$
G_\alpha(x,y,t)=\frac{\partial A_\alpha}{\partial x}(x,y,t).
$$

Moreover, when $\alpha>\frac{1}{2},$ $G_\alpha(\cdot, \cdot, t)$ is for $t\neq 0$ a continuous function in $\R^2$, bounded when $\alpha \geq 2.$ \footnote{Contrary to what is claimed (but not used!)  in \cite{MST1}, $G_\alpha$ is not a bounded function when $\alpha >1/2$.}
In addition, for $t\neq 0$, $y\in\R$, $\varphi\in L^1(\R^2)$,

$$\lim_{|x|\rightarrow \infty}(A_\alpha\star\varphi)(x,y,t)=0.$$

As a consequence, the solution of\/ {\rm(\ref{lfKP})} with data $\varphi\in L^1(\R^2)$
is given by 

$$u(\cdot,\cdot,t)\equiv U_\alpha(t)\varphi:=G_\alpha\star \varphi$$

and

$$u(\cdot,\cdot,t)=\frac{\partial}{\partial x}\bigl( A_\alpha\star \varphi\bigr).$$

One therefore has

$$\int_{-\infty}^{\infty}u(x,y,t)\, dx =0\quad \forall y\in\R,\,\,\, \forall t\neq
0\, $$

in the sense of generalized Riemann integrals.
\end{theorem}

\begin{remark}
It is worth noticing that the result of Theorem\/~{\rm\ref{thm2.1}} is related to the
infinite speed of propagation of the KP free evolutions.

The proof of Theorem\/~{\rm\ref{thm2.1}} contains implicitly, when $\alpha\geq 2,$  an $L^1-L^\infty$ estimate with time decay of order $1/t$ on the fundamental solution of \eqref{fKP} leading  \lq\lq for free\rq\rq \, to Strichartz estimates. We will see below how to obtain  Strichartz estimates with loss in the general case $\alpha>0.$

Also, as noticed in \cite{MST1}, the result above extends with a few technicalities to non-homogeneous symbols that behave as $|\xi|^\alpha$ at infinity as for instance in the case of the ILW-KP equation.
\end{remark}

We now turn to the constraint problem in the nonlinear case, again following \cite{MST1}. After a change of frame we can eliminate the $u_x$ term and reduce the Cauchy problem for (\ref{fKP}) to
\begin{equation}\label{2.3}
(u_t+uu_x-D^\alpha_x u_x)_x+\kappa u_{yy}=0,\quad u(x,y,0)=\varphi(x,y).
\end{equation}
In order to state the result concerning (\ref{2.3}), for $k\in \N$, we denote by $H^{k,0}(\R^2)$ the Sobolev space
of $L^2(\R^2)$ functions $u(x,y)$ such that $\partial_{x}^{k}u\in L^2(\R^2)$.

\begin{theorem}\label{thm2.2} \cite{MST1}
Assume that $\alpha>1/2$. 
Let $\varphi\in L^{1}(\R^2)\cap H^{2,0}(\R^2)$ and 
\begin{equation}\label{new}
u\in C([0,T]\, ;\, H^{2,0}(\R^2))
\end{equation}
be a distributional solution of\/ {\rm(\ref{2.3})}.
Then, for every $t\in (0,T]$, $u(t,\cdot,\cdot)$ is a continuous function of
$x$ and $y$ which satisfies
\[
\int_{-\infty}^{\infty}u(x,y,t)\,dx =0\quad \forall y\in \R,\,\,\, \forall t\in (0,T]
\]
in the sense of generalized Riemann integrals.
Moreover, $u(x,y,t)$ is the
derivative with respect to $x$ of a $C^1_x$ continuous function
which vanishes as
$x\rightarrow \pm \infty$ for every fixed $y\in\R$ and $t\in [0,T]$.
\end{theorem}

\begin{remark}
Solutions in the class $H^{k,0}(\R^2)$ can be obtained when $\alpha=2, \kappa =1,$ that is for the classical KP-II equation and also when $\alpha>\frac{4}{3}$ for the fKP-II equation (see \cite{Ha}).We do not know if they exist 
in the range $0<\alpha <\frac{4}{3}.$

On the other hand, the structure of a solution corresponding to a gaussian initial data (thus not satisfying the zero mass constraint) is illustrated by the numerical simulations \cite{KSM} for the usual KP equation.

 \end{remark}

\subsection{Global weak solutions for fKP-I equations}
Here the idea is to use the conservation of the $L^2$ norm and of the hamiltonian  to construct global weak solutions of the fKP-I equation in the $L^2$ subcritical case. This is well known for the standard KP-I equation (see \cite{Tom}). The extension to fKP-I is straightforward and we indicate it for the sake of completeness.

More precisely one has

\begin{theorem}
Assume that $\kappa=-1$ (fKP-I) and that $\alpha>\frac{3}{2}.$ Let $u_0\in X^{\frac{\alpha}2}.$ Then there exists a global weak solution $u\in L^\infty(\R : X^{\frac{\alpha}2})$ of \eqref {fKP}. The same result holds true for $\alpha =\frac{3}{2}$ provided $\|u_0\|_{L^2}$ is sufficiently small.
\end{theorem}

\begin{proof}
It results from a standard compactness method. The key point is to notice that when $\alpha >\frac{3}{2},$ (and when $\alpha =\frac{3}{2}$ for a sufficiently small $L^2$ norm), the $X^{\frac{\alpha}2}$ norm control a priori the $L^3$ norm as consequence of Lemma \ref{BLT} and the conservation of the $L^2$ norm.
\end{proof}


\section{Ill-posedness issues for fKP}

\subsection{Ill-posedness issues for fKP-II}
We first prove that the fKP-II equations are \emph{quasilinear} as soon as $\alpha < 4/3$ in the sense that the local Cauchy problem cannot be solved, for initial data in any isotropic or anisotropic Sobolev space, by a Picard iterative scheme based on the Duhamel formula. This proves in particular the  result obtained by Hadac in \cite{Ha} is (almost) sharp with respect to the condition $\alpha >\frac43$.

\begin{theorem} \label{illposed-fKPII} 
Assume $\kappa=1$ (fKP-II). Let $\alpha \in (0,\frac43)$ and $(s_1,s_2) \in \mathbb R^2$ (resp. $s \in \mathbb R$). 
Then, there exists no $T>0$ such that \eqref{fKP} admits a unique local solution defined on the time interval $[0,T]$ and such that its flow-map 
\begin{displaymath}
S_t: u_0 \mapsto u(t), \quad t \in (0,T]
\end{displaymath}
is $C^2$ differentiable at zero from $H^{s_1,s_2}(\mathbb R^2)$ to $H^{s_1,s_2}(\mathbb R^2)$, (resp. from $X^s(\mathbb R^2)$ to $X^s(\mathbb R^2)$).
\end{theorem} 

\begin{remark} We will write the proof for $H^{s_1,s_2}(\mathbb R^2)$, but it will be clear that this also works with $X^s(\mathbb R^2)$ instead.  
\end{remark}

\begin{proof}
Our goal is to prove that the inequality 

\begin{equation}\label{picard}
\Big\|\int\limits_0^t U_{\alpha}(t-t')(U_{\alpha}(t')\phi_1 \p_x U_{\alpha}(t')\phi_2)(x,y)\,dt'\Big\|_{H^{s_1,s_2}} \lesssim \|\phi_1\|_{H^{s_1,s_2}}\|\phi_2\|_{H^{s_1,s_2}} ,
\end{equation}
does not hold for any $\phi_1, \phi_2 \in H^{s_1,s_2}(\R^2)$ and any $s_1, s_2\in\R$. This in particular implies
 that the data-solution map is not $C^2$.

To show this we will follow the arguments in \cite{MST} and \cite{rv}.  More precisely, it would be enough to construct sequences of functions $\phi_{i,N}$, $i=1,2$, such that
for any $s_1, s_2 \in\R$ it holds
\begin{equation}\label{seq}
\|\phi_{i,_N}\|_{H^{s_1,s_2}}\le C
\end{equation}
and 
\begin{equation}\label{inequality}
\underset{N\to\infty}{\lim} \Big\|\int\limits_0^t U_{\alpha}(t-t')(U_{\alpha}(t')\phi_{1,N}\p_x U_{\alpha}(t')\phi_{2,N})(x,y)\,dt'\Big\|_{H^{s_1,s_2}} =  + \infty.
\end{equation}

First  we observe that
\begin{equation*}
\begin{split}
\Big\{\int\limits_0^t &U_{\alpha}(t-t')(U_{\alpha}(t')\phi_1 \p_x U_{\alpha}(t')\phi_2)(x,y)\,dt'\Big\}^{\widehat{\hskip 5pt}}(\xi,\eta)\\
&=
\int\limits_{\R^2} e^{it\omega(\xi,\eta)}\frac{e^{it\Omega^{\alpha}(\xi_1,\xi_2,\eta_1,\eta_2)} -1}{\Omega^{\alpha}(\xi_1,\xi_2,\eta_1,\eta_2) }\;\xi\;
 \widehat{\phi}_1(\xi-\xi_1,\eta-\eta_1)\,\widehat{\phi_2}(\xi_1,\eta_1)
\,d\xi_1 d\eta_1\\
\end{split}
\end{equation*}
where the resonant function
\begin{equation*}
 \Omega^{\alpha}(\xi_1,\xi_2,\eta_1,\eta_2)=(|\xi|^{\alpha}\xi-\frac{\eta^2}{\xi})-(|\xi_1|^{\alpha}\xi_1-\frac{\eta^2_1}{\xi_1})-(|\xi_2|^{\alpha}\xi_2-\frac{\eta^2_2}{\xi_2}) \, ,
 \end{equation*}
 with $\xi=\xi_1+\xi_2$ and $\eta=\eta_1+\eta_2$. 
 
 Arguing as in \cite{rv} for the fZK equation, we choose
 \begin{equation*}
 \begin{cases}
 \widehat{\phi_{1,N}}(\xi_1,\eta_1)= \mathbbm{1}_{D_1}\,\gamma^{-\frac12-\frac{\epsilon}2} \hskip35pt \text{with}\hskip10pt D_1=[\frac{\gamma}2,\gamma]\times[\gamma^{\epsilon}, 2\gamma^{\epsilon}]\\
  \widehat{\phi_{2,N}}(\xi_2,\eta_2)= \mathbbm{1}_{D_2} \,N^{-s_1}\gamma^{-\frac12-\frac{\epsilon}2} \hskip10pt \text{with}\hskip10pt D_2=[N, N+\gamma]\times[-\gamma^{\epsilon}, -\frac{\gamma^{\epsilon}}4]
  \end{cases}
  \end{equation*}
where $N\gg1$ and $0<\gamma\ll 1$ and $\frac12<\epsilon$ to be determined later. We observe that $\phi_{1,N}$ and $\phi_{2,N}$ satisfy \eqref{seq}.

Moreover, notice that
\begin{equation*}
\Omega^{\alpha}(\xi_1,\xi_2,\eta_1,\eta_2)=
\underset{\Gamma_1^{\alpha}(\xi_1,\xi_2,\eta_1,\eta_2)}{ \underbrace{|\xi_1+\xi_2|^{\alpha}(\xi_1+\xi_2)-|\xi_1|^{\alpha}\xi_1-|\xi_2|^{\alpha}\xi_2}}
+\underset{\Gamma_2^{\alpha}(\xi_1,\xi_2,\eta_1,\eta_2)}{\underbrace{\frac{(\eta_1\xi_2-\eta_2\xi_1)^2}{(\xi_1+\xi_2)\xi_1 \xi_2}}}
\end{equation*}

We start estimating the contribution given by $\Gamma_1(\xi_1,\xi_2,\eta_1,\eta_2)$. We can use the mean value theorem to deduce the existence of $\theta \in [\xi_2,\xi_2+\xi_1]$
such that
\begin{equation*}
|\xi_1+\xi_2|^{\alpha}(\xi_1+\xi_2)-|\xi_2|^{\alpha}\xi_2= (\alpha+1)|\xi_1| |\theta|^{\alpha}.
\end{equation*}
This leads to
\begin{equation}\label{mvt1}
\big||\xi_1+\xi_2|^{\alpha}(\xi_1+\xi_2)-|\xi_2|^{\alpha}\xi_2\big|\sim \g N^{\alpha}.
\end{equation}
On the other hand, $|\xi_1|\sim \g ={\rm o}(N)$ from which follows that
\begin{equation}\label{mvt2}
||\xi_1|^{\alpha}\xi_1|\sim \g^{\alpha+1}.
\end{equation}
Combining \eqref{mvt1} and \eqref{mvt2} we obtain
\begin{equation}\label{mvt3}
|\Gamma_1^{\alpha}(\xi_1,\xi_2,\eta_1,\eta_2)| \sim \g\,N^{\alpha}.
\end{equation}

Next  we estimate $\Gamma_2(\xi_1,\xi_2,\eta_1,\eta_2)$. Using that $\xi_2\in [N,N+\g]$ and $\eta_2\in [-\g^{\epsilon}, -\frac{\g^{\epsilon}}4]$,
we obtain that
\begin{equation*}
N\g^{\epsilon} \le \eta_1\,\xi_2\le 2(N+\g)\g^{\epsilon} \hskip15pt \text{and}\hskip15pt  -\frac{\g^{1+\epsilon}}2\le \eta_2\,\xi_1\le 
\frac{\g^{1+\epsilon}}4.
\end{equation*}
Hence $(\eta_1\xi_2-\eta_2\xi_1)\sim (N\g^{\epsilon})^2$. Since $|(\xi_1+\xi_2)\xi_1\xi_2|\sim N^2\g $ we conclude
that
\begin{equation}\label{mvt4}
|\Gamma_2^{\alpha}(\xi_1,\xi_2,\eta_1,\eta_2)|\sim \g^{2\epsilon-1}.
\end{equation}

Thus, we want to satisfy the conditions
$$
\gamma N^{\alpha}\ll 1 \iff \gamma =N^{-\alpha-\theta} \hskip5pt \forall\,\theta>0
$$
and 
$$
\gamma^{2\epsilon-1}\ll 1 \iff  \epsilon >1/2.
$$
Therefore taking $\gamma =N^{-\alpha-\theta}$ and $\epsilon=\frac12+\theta$, $ \forall \theta>0$, we have $|\Omega^{\alpha}|\ll 1$.

Notice that $\|\phi_i\|_{H^{s_1,s_2}}\lesssim 1$, $i=1,2$. Then the inequality \eqref{picard} holds if and only if
\begin{equation*}
\begin{split}
\frac{NN^{s_1}\gamma^{ 1+\epsilon}\gamma^{\frac12+\frac{\epsilon}2}}{\gamma^{1+\epsilon}N^{s_1}} \lesssim 1
\iff & N^{1-\frac34\alpha-{\tilde{\theta}}} \lesssim 1, \hskip10pt \text{where}\hskip5pt  \tilde{\theta}=\theta(\frac12+\frac{\alpha}2+\frac{\theta}2)>0\\
\iff & \alpha \ge \frac43.
\end{split}
\end{equation*}

From this we can conclude that \eqref{picard} does not hold if $0<\alpha<\frac43$ for any $s_1, \, s_2\in\R$. Therefore the IVP associated to fKPII is ill-posed for $\alpha\in(0,\frac43)$ and any $s_1, \, s_2 \in\R$. 
\end{proof}

\begin{remark} \label{IPfKPII}
 1. The result is sharp in the sense that when $\alpha >\frac{4}{3}$ it is proven in \cite{Ha} that the Cauchy problem for fKP-II can be solved by a iterative method.
2. The same argument applies to  the fKP-I equation when $\alpha <\frac{4}{3}$ but we will see below that the result holds true for any $\alpha\in (0,2]$ in this case.
\end{remark}

\subsection{Ill-posedness issue for fKP-I}
In this subsection, we prove that the fKP-I equations are \emph{quasilinear} for $0<\alpha \le 2$ in the sense that the local Cauchy problem cannot be solved, for initial data in any isotropic or anisotropic Sobolev space, by a Picard iterative scheme based on the Duhamel formula. This fact has already been established in \cite{MST} for the KP-I equation itself ($\alpha=2$).

\begin{theorem} \label{illposed-fKPI} 
Assume $\kappa=-1$ (fKP-I). Let $\alpha \in (0,2]$ and $(s_1,s_2) \in \mathbb R^2$ (resp. $s \in \mathbb R$). 
Then, there exists no $T>0$ such that \eqref{fKP} admits a unique local solution defined on the time interval $[0,T]$ and such that its flow-map 
\begin{displaymath}
S_t: u_0 \mapsto u(t), \quad t \in (0,T]
\end{displaymath}
is $C^2$ differentiable at zero from $H^{s_1,s_2}(\mathbb R^2)$ to $H^{s_1,s_2}(\mathbb R^2)$, (resp. from $X^s(\mathbb R^2)$ to $X^s(\mathbb R^2)$).
\end{theorem} 

\begin{remark}
We consider only the cases where $0<\alpha \le 2$, but our result for the fKPI case probably holds for $0<\alpha<\alpha_0$, for some $\alpha_0>2$. As aforementioned, the fifth order  ($\alpha =4$) KP-I equation is semilinear (see \cite{ST}).
\end{remark}

\begin{proof}
Since the argument employed is similar to the fKPII case we will give a sketch of the
main differences between both cases. Note that due to Remark \ref{IPfKPII}, we may assume that $\frac43\le \alpha \le 2$.
As in the previous case we have that
\begin{equation*}
\begin{split}
\Big\{\int\limits_0^t &U_{\alpha}(t-t')(U_{\alpha}(t')\phi_1 \p_x U_{\alpha}(t')\phi_2)(x,y)\,dt'\Big\}^{\widehat{\hskip 5pt}}(\xi,\eta)\\
&=\int\limits_{\R^2} e^{it\omega(\xi,\eta)}\frac{e^{it\Omega^{\alpha}(\xi_1,\xi_2,\eta_1,\eta_2)} -1}{\Omega^{\alpha}(\xi_1,\xi_2,\eta_1,\eta_2) }\;\xi\;
 \widehat{\phi}_1(\xi-\xi_1,\eta-\eta_1)\,\widehat{\phi_2}(\xi_1,\eta_1)
\,d\xi_1 d\eta_1
\end{split}
\end{equation*}
where the resonant function
\begin{equation*}
\begin{split}
 \Omega^{\alpha}(\xi_1,\xi_2,\eta_1,\eta_2)&=\Big(|\xi_1+\xi_2|^{\alpha}(\xi_1+\xi_2)+\frac{(\eta_1+\eta_2)^2}{(\xi_1+\xi_2)}\Big)\\
 &\quad-\Big(|\xi_1|^{\alpha}\xi_1+\frac{\eta^2_1}{\xi_1}\Big)-\Big(|\xi_2|^{\alpha}\xi_2+\frac{\eta^2_2}{\xi_2}\Big)\\
 &=\underset{\Gamma^{\alpha}_1(\xi_1,\xi_2,\eta_1,\eta_2)}{\underbrace{|\xi_1+\xi_2|^{\alpha}(\xi_1+\xi_2)
 -|\xi_1|^{\alpha}\xi_1-|\xi_2|^{\alpha}\xi_2}}
 -\underset{\Gamma^{\alpha}_2(\xi_1,\xi_2,\eta_1,\eta_2)}{\underbrace{\frac{(\eta_1\xi_2-\eta_2\xi_1)^2}{(\xi_1+\xi_2)\xi_1 \xi_2}}}
 \end{split}
 \end{equation*}
 
 We choose now following \cite{MST} 
\begin{equation*}
\widehat{\phi}_1(\xi_1,\eta_1)= \mathbbm{1}_{D_1}\,\gamma^{-\frac32} \quad \text{with} \quad D_1=[\frac{\g}{2},\g]\times [-\sqrt{1+\alpha}\,\g^{2},  \sqrt{1+\alpha}\,\g^2] 
\end{equation*}
and 
\begin{equation*}
\widehat{\phi}_2(\xi_2,\eta_2)= \mathbbm{1}_{D_2} \, \gamma^{-\frac32} N^{-s_1-(1+\frac{\alpha}2) s_2}
\end{equation*}
with 
\begin{equation*}
D_2=[N,N+\g]\times [\sqrt{1+\alpha}\, N^{\frac{\alpha+2}{2}},  \sqrt{1+\alpha} \,N^{\frac{\alpha+2}{2}}+\g^2]
\end{equation*}
where $N\gg1$ and $0<\gamma\ll 1$ to be determined later. It is easy to verify that $\;\|\phi_i\|_{H^{s_1,s_2}}\lesssim 1$, $i=1,2$.

Using that $(\xi_1,\eta_1)\in D_1$ and  $(\xi_2,\eta_2)\in D_2$ we deduce the following estimate
\begin{equation}\label{gama1}
\Gamma^{\alpha}_1(\xi_1,\xi_2,\eta_1,\eta_2)=(1+\alpha)N^{\alpha}\gamma +
{\rm O}(\gamma^2 N^{\alpha-1}).
\end{equation}

On the other hand,
\begin{equation}\label{gama2}
\Gamma^{\alpha}_2(\xi_1,\xi_2,\eta_1,\eta_2)=(1+\alpha)N^{\alpha}\g
+{\rm O}(N^{\frac{\alpha}{2}}\g^2).
\end{equation}

Using \eqref{gama1} and \eqref{gama2} we require
\begin{equation*}
|\Omega^{\alpha}| \lesssim \Big| N^{\alpha-1}\g^2\Big|+\Big|N^{\frac{\alpha}{2}}\g^2\Big|\ll 1.
\end{equation*}
Hence
$$
N^{\frac{\alpha}{2}}\g^2\ll 1 \iff \g=N^{-\frac{\alpha}{4}-\theta} \hskip10pt \forall \;\theta>0.
$$
This implies that
$$
N^{\alpha-1}\g^2=N^{\alpha-1} N^{-\frac{\alpha}{2}-2\theta} = N^{\big(\frac{\alpha}{2}-1\big)-2\theta}\ll 1 \iff \alpha \le2 \, .
$$

Therefore using the properties of the functions $\phi_i$, $i=1,2$, the inequality \eqref{picard} holds if and only if
\begin{equation*}
\begin{split} 
\frac{N N^{s_1+(1+\frac{\alpha}{2})s_2} \g^3 \g^{\frac32}}{\g^3N^{s_1+(1+\frac{\alpha}{2})s_2}}
=N  \g^{\frac32} \lesssim 1 
\iff& N^{1-\frac{3\alpha}{8}-\frac{3\theta}{2}}\lesssim 1\\
\iff& \alpha \ge \frac{8}{3}.
\end{split}
\end{equation*}

Thus the inequality \eqref{picard} does not satisfy if $0<\alpha\le 2$. We deduce that the IVP associated to the fKP-I is ill-posed for $0<\alpha\le2$ 
in $H^{s_1,s_2}(\R)$ for any $s_1, s_2\in\R$.
\end{proof}

\begin{remark}
As aforementioned the results above prove that distinguishing between {\it semilinear} and {\it quasilinear} is not obvious for nonlinear dispersive equations since that depends on a subtle interaction between the linear and nonlinear parts. Recall also that the Benjamin-Ono, equation is {\it quasilinear} (\cite{MST}) while the modified Benjamin-Ono (that is with the nonlinearity $u^2u_x$) is {\it semilinear} (\cite{KPV3}).
\end{remark}

\section{Improved well-posedness}

We use here the dispersive properties of the free group to improve the \lq\lq standard\rq\rq \, local well-posedness theory of fKP when $\alpha>0.$ We follow the strategy used by Kenig in \cite{K} for the classical KP-I equation (see also \cite{LPS2} for fractional KdV equations).

We consider the IVP associated to the fKP equation \eqref{fKP}, without distinguishing th KP-II case ($\kappa=1$) from the KP-I case ($\kappa =-1$)
\medskip

Our main result states that for any $\alpha \in (0,2]$, the IVP \eqref{fKP} is locally well-posed in $X^s(\mathbb R^2)$ for $s>2-\frac{\alpha}4$.

\begin{theorem} \label{maintheo}
Let $0<\alpha\le 2$. Define $s_{\alpha}:=2-\frac{\alpha}{4}$ and assume that $s>s_{\alpha}$. Then,
for any $u_0 \in X^s(\mathbb R^2)$, there exist a positive time $T=T(\|u_0\|_{X^s})$ (which can be chosen as a nondecreasing function of its argument) and a unique solution $u$ to the IVP \eqref{fKP} in the class 
\begin{equation} \label{maintheo.1}
C([0,T] : X^s(\mathbb R^2)) \cap L^1((0,T) : W^{1,+\infty}_x(\mathbb R^2)) \,  .
\end{equation}

Moreover, for any $0<T'<T$, there exists a neighbourhood $\mathcal{U}$ of $u_0$ in $X^s(\mathbb R^2)$ such that the flow map data solution 
\begin{displaymath}
S_{T'}^s: \mathcal{U} \to C([0,T'] : X^s(\mathbb R^2)) , \quad v_0 \mapsto v \, ,
\end{displaymath}
is continuous.
\end{theorem}

\begin{remark} The result in the case $\alpha=2$ corresponding to the KPI equation was already obtained by Kenig in \cite{K}. It is a challenging problem to lower the exponent $s_\alpha.$ Recall that for instance, the standard KP-I equation is well posed in the corresponding energy space (\cite{IKT}).
\end{remark}

Actually, it turns out that our proof also works for more general non-homogeneous  KP type equations. Let us describe those equations more precisely.

\begin{definition} \label{def_symb}
For $\alpha>0$, let $\mathcal{L}_{\alpha+1}$ be the Fourier multiplier defined by 
\begin{displaymath}
\mathcal{F}(\mathcal{L}_{\alpha+1}f)(\xi,\eta)=iw_{\alpha+1}(\xi)\mathcal{F}(f)(\xi,\eta) \, ,
\end{displaymath}
where  $w_{\alpha+1}$  is an odd real-valued function belonging to $C^1(\mathbb R) \cap C^{\infty}(\mathbb R \setminus \{0\})$  satisfying 
\begin{equation} \label{hypothesis_w1}
|w_{\alpha+1}(\xi)| \lesssim 1 \, , \quad \forall \, |\xi| \le \xi_0 \, ,
\end{equation}
and
\begin{equation} \label{hypothesis_w2}
|\partial^{\beta} w_{\alpha+1}(\xi)| \sim |\xi|^{\alpha+1-\beta}, \quad \forall \, |\xi| \ge \xi_0, \quad \forall \, \beta=0,1,2 \, ,
\end{equation}
for some fixed $\xi_0 >0$.
\end{definition}

\begin{remark} The following symbols satisfy the conditions \eqref{hypothesis_w1} and \eqref{hypothesis_w2}: 
\begin{enumerate}
\item{} the \emph{pure power symbol} $w_{\alpha+1}(\xi)=|\xi|^{\alpha}\xi$ corresponding to the fractional dispersive operator $\mathcal{L}_{\alpha+1}=D_x^{\alpha}\partial_x$ with $\alpha>0$.

\item{} the \emph{surface tension Whitham symbol} $\left( \frac{\tanh{\xi}}{\xi} \right)^{\frac12}(1+b \xi^2)^{\frac12}\xi$, with $b>0$ corresponding to $\alpha=\frac12$.

\item{} the \emph{intermediate long wave symbol} $\coth(\xi)|\xi|\xi$ corresponding to $\alpha=1$.
\end{enumerate}

Our result however does not include the Full Dispersion KP equation \eqref{FDbis}. A first difficulty arises there when proving the $L^1-L^\infty$ estimate because the study of the underlying oscillatory integral (see Lemma 4.7 below) cannot be reduced to a one-dimensional one as it is the case for the fKP equations.
\end{remark}

We are interested in the Cauchy problem associated to non-homogeneous KP type equations of the form
\begin{equation} \label{nh_fKP}
u_t+uu_x-\mathcal{L}_{\alpha+1}u \pm\partial_x^{-1}u_{yy}=0 \, .
\end{equation}

\begin{theorem} \label{nh_theo}
Let $0<\alpha \le 2$ and $\mathcal{L}_{\alpha+1}$ be a Fourier multiplier as in Definition \ref{def_symb}. Define $s_{\alpha}:=2-\frac{\alpha}{4}$ and assume that $s>s_{\alpha}$. Then the Cauchy problem associated to \eqref{nh_fKP} is locally well-posed in $X^s(\mathbb R^2)$ in the sense of Theorem \ref{maintheo}.
\end{theorem}

In the following we will give the proof of Theorem \ref{maintheo} and will only indicate the changes to prove Theorem \ref{nh_theo} when necessary.

\bigskip

\subsection{Commutator and interpolation estimates}
To obtain estimates for the nonlinear terms, the following Leibniz rules for fractional derivatives will be needed in Subsection \ref{nonlinest}. 

\begin{lemma} \label{Kato-Ponce}
For $s >1$, it holds that 
\begin{equation} \label{Kato-Ponce.1}
\big\| [J^s_x,f]g \big\|_{L^2(\mathbb R)} \lesssim \|\partial_xf\|_{L^{q_1}(\mathbb R)}\|J^{s-1}_xg\|_{L^{p_1}(\mathbb R)}
+\|J^s_xf\|_{L^{p_2}(\mathbb R)}\|g\|_{L^{q_2}(\mathbb R)} \, ,
\end{equation}
where $\;\frac{1}{p_i}+\frac{1}{q_i}=\frac12$, $1<p_1, p_2< \infty$ and $1< q_1, q_2\le \infty$ and 
\begin{displaymath}
[J^s_x,f]g = J^s_x(fg)-fJ^s_xg \, .
\end{displaymath}
\end{lemma}
Lemma \ref{Kato-Ponce} was proved by Kato and Ponce in \cite{kp}.

\begin{lemma}\label{frac-deriv} \hskip10pt
\begin{itemize}
\item[(a)] For $\sigma \in (0,1)$, it holds that
\begin{equation}\label{Leibniz-1d}
\|D^{\sigma}_x(fg)\|_{L^2(\R)} \lesssim  \|D^{\sigma}_xf\|_{L^{p_1}(\R)}\|g\|_{L^{q_1}(\R)}+\|D^{\sigma}_xg\|_{L^{p_2}(\R)}\|f\|_{L^{q_2}(\R)}
\end{equation}
with $\;\frac{1}{p_i}+\frac{1}{q_i}=\frac12$, $1<p_1, p_2\le  \infty$ and $1< q_1, q_2\le \infty$.\\
\item[(b)] For $\sigma, \beta\in (0,1)$, it holds that
\begin{equation}\label{leibniz-2d}
\begin{split}
\|D^{\sigma}_xD^{\beta}_y (fg)\|_{L^2(\R^2)} &\lesssim \|f\|_{L^{p_1}(\R^2)}\|D^{\sigma}_xD^{\beta}_yg\|_{L^{q_1}(\R^2)}+
 \|D^{\sigma}_xD^{\beta}_y f\|_{L^{p_2}(\R^2)}\|g\|_{L^{q_2}(\R^2)}\\
&\hskip15pt+ \|D^{\beta}_y f\|_{L^{p_3}(\R^2)}\|D^{\sigma}_x g\|_{L^{q_3}(\R^2)}+\|D^{\sigma}_x f\|_{L^{p_4}(\R^2)}\|D^{\beta}_y g\|_{L^{q_4}(\R^2)} \, ,
\end{split}
\end{equation}
where $\;\frac{1}{p_i}+\frac{1}{q_i}=\frac12$, $1<p_i\le \infty$,  $1<q_i\le \infty$, $i=1, 2, 3, 4$.
\end{itemize}
\end{lemma}

The proof of (a) can be seen in Kenig, Ponce and Vega \cite{KPV}. The estimate (b) was proved by Muscalu, Pipher, Tao and Thiele in \cite{mptt}. For more details see \cite{K}.

\medskip
To close the argument some \lq\lq interpolated\rq\rq \, estimates will be useful. 
\begin{lemma}\label{interpolated}
Let $0<\alpha \le 2$ and $s_{\alpha}:=2-\frac{\alpha}4$.
\begin{itemize}
\item[(a)] For $0<\delta<\frac{\alpha}4$, we have
\begin{equation}\label{kenig-(2.5)}
\|D^{s_{\alpha}-1+\delta}_xu\|_{L^{\infty}_{xy}}\lesssim \|u\|_{L^{\infty}_{xy}}+\|\partial_xu\|_{L^{\infty}_{xy}}.
\end{equation}
\item[(b)]  If $\delta_0$ is a positive constant chosen small enough, then the following holds true. There exist 
\begin{displaymath}
\left\{\begin{array}{l}  2<p_1,q_1<\infty \\ 1<r_1,s_1 <\infty \end{array}\right. \quad \text{with} \quad  \frac1{p_1}+\frac1{q_1}=\frac12, \quad \frac1{r_1}+\frac1{s_1}=1 \, ,
\end{displaymath}
$0<\theta<1$ and $0<\delta_1=\delta_1(\delta_0,\theta) \ll 1$ such that 
\begin{equation}\label{kenig-(2.6)}
\|D^{s_{\alpha}-1+\delta}_x\partial_x u\|_{L^{s_1}_TL^{q_1}_{xy}}\lesssim \|\partial_xu\|_{L^1_TL^{\infty}_{xy}}^{\theta}
\|J^{s_{\alpha}+\delta_0}_xu\|_{L^{\infty}_TL^2_{xy}}^{1-\theta},
\end{equation}
and
\begin{equation}\label{kenig-(2.7)}
\|D^{\delta}_yu\|_{L^{r_1}_TL^{p_1}_{xy}}\lesssim \|u\|_{L^1_TL^{\infty}_{xy}}^{1-\theta}\big(\|D^{\frac12}_yu\|_{L^{\infty}_TL^2_{xy}}
+\|u\|_{L^{\infty}_TL^2_{xy}}\big)^{\theta} ,
\end{equation}
for all $0<\delta<\delta_1$.
\end{itemize}
\end{lemma}

These estimates were proved by Kenig in \cite{K} (see estimates (2.5), (2.6), and (2.7)). For the sake of completeness, we will give the proofs in the appendix.

\subsection{Linear Estimates}

Consider the linear IVP
\begin{equation}\label{linearfKP}
\begin{cases}
\partial_t u -D^{\alpha}_x\partial_x u+\kappa\partial_x^{-1}\partial_y^2u=0,\;\;\;(x,y)\in\R^2, t>0,\\
u(x,y,0)=u_0(x,y)
\end{cases}
\end{equation}
where $\kappa=\pm1$ and whose solution is given by
\begin{equation}
u(x,y,t)=U_{\alpha}(t)u_0(x,y):= \big( e^{it(|\xi|^{\alpha}\xi-\kappa\frac{\eta^2}{\xi})}\widehat{u}_0(\xi,\eta)\big)^{\vee}(x,y).
\end{equation}

\smallskip
Solutions of the linear problem \eqref{linearfKP} satisfy the following decay estimate
\begin{lemma} \label{decay} For $\alpha\in (0,2] $, it holds that
\begin{equation}
\|D^{\frac{\alpha}{2}-1}_x U_{\alpha}(t)\phi\|_{L^{\infty}(\R^2)}\le c\,|t|^{-1}\|\phi\|_{L^1}.
\end{equation}
\end{lemma}
\begin{proof}  We can always assume without loss of generality that $\kappa=-1$ and $t>0$. We observe that
\begin{equation*} 
U_{\alpha}u_0(x,y)= I(\cdot,\cdot,t)\ast u_0(x,y) \, ,
\end{equation*}
where
\begin{equation*}
 I(x,y,t)=(2\pi)^{-1} \int_{\R^2} e^{it(|\xi|^{\alpha}\xi+\frac{\eta^2}{\xi})} e^{i(x\xi+y\eta)}\,d\xi d\eta \, .
 \end{equation*}
 
Next we will study the decay properties of the following oscillatory integral,
\begin{equation}\label{oi1}
D^{\frac{\alpha}{2}-1}_x I(x,y,t)=(2\pi)^{-1} \int_{\R^2}|\xi|^{\frac{\alpha}{2}-1} e^{it(|\xi|^{\alpha}\xi+\frac{\eta^2}{\xi})} e^{i(x\xi+y\eta)}\,d\xi d\eta \, .
\end{equation}

Performing the change of variables  $\eta'=t^{\frac12}|\xi|^{-\frac{1}{2}}\,\eta$ yields
\begin{displaymath}
D^{\frac{\alpha}{2}-1}_x I(x,y,t)=\frac{c}{t^{\frac12}}
\int_{\R_{\xi}}|\xi|^{\frac{(\alpha-1)}{2}} \Big (\int_{\R_{\eta}} e^{it^{-\frac{1}{2}}|\xi|^{\frac12}y\eta+i\sgn(\xi)\eta^2)} d\eta \Big) e^{i(x\xi+t|\xi|^{\alpha}\xi)}\,d\xi \, .
\end{displaymath}
It follows then by using the formula $\big(e^{i\delta \eta^2} \big)^{\vee}(x)=c|\delta|^{-\frac12}e^{i\sgn(\delta)\frac\pi4}e^{i\delta^{-1}x^2}$, for any $\delta \in \mathbb R$, $\delta \neq 0$, that
\begin{align}\label{oi2}
D^{\frac{\alpha}{2}-1}_x I(x,y,t)&=\frac{c}{t^{\frac12}} \int_{\R} |\xi|^{\frac{(\alpha-1)}{2}} \,e^{i\sgn(\xi)\frac{\pi}{4}} e^{i\frac{y^2\xi}{4t}}e^{i(x\xi+t|\xi|^{\alpha}\xi)}\,d\xi \nonumber \\
&=\frac{c}{t}\int_{\R} |\xi|^{\frac{(\alpha-1)}{2}}  \, e^{i\sgn(\xi)\frac{\pi}{4}}  
\,\exp\Big(i\xi\Big(\frac{x}{t^{\frac{1}{\alpha+1}}}+\frac{y^2}{4t^{\frac{\alpha+2}{\alpha+1}}}\Big)\Big) e^{i\xi|\xi|^{\alpha}}\,d\xi \, .
\end{align}
Define
\begin{equation}
J(\l)= \int_{\R}  |\xi|^{\frac{(\alpha-1)}{2}} \, e^{i\sgn(\xi)\frac{\pi}{4}}  
\,e^{i\l\xi} e^{i\xi|\xi|^{\alpha}}\,d\xi.
\end{equation}

To complete the proof we need to establish that $J(\lambda)$ is bounded for any $\lambda\in\R$, as soon as $\alpha>0$. This can be done using Kenig, Ponce,
Vega theory in \cite{KPV1} where they showed that the decay result  is sharp (see also \cite{LPS2}). This gives us the desired result.
\end{proof}

\begin{remark}
Lemma \ref{decay} generalizes the decay estimate for the KP equations (corresponding to $\alpha=2$) obtained by Saut in \cite{S}.
\end{remark}

We explain how to obtain a similar decay estimate for the linear solutions associated to \eqref{nh_fKP}. Since the symbol $w_{\alpha+1}$ is non-homogeneous, we need to distinguish the low and high frequencies in $x$.
\begin{lemma} \label{nh_decay} 
For $\alpha\in (0,2] $, let $e^{t(\mathcal{L}_{\alpha+1}\mp \partial_x^{-1}\partial_y^2)}$ be the unitary group associated to \eqref{nh_fKP} and let $\chi \in C_0^{\infty}(|\xi| <2)$. We define by $Q_{\le \xi_0}$ and $Q_{\ge  \xi_0}$ the Fourier multipliers of respective symbols $\chi(\xi/\xi_0)$ and $\psi_0(\xi)=1-\chi(\xi/\xi_0)$. Then,
\begin{equation} \label{nh_decay.1}
\| Q_{ \le \xi_0}e^{t(\mathcal{L}_{\alpha+1}\mp \partial_x^{-1}\partial_y^2)}\phi\|_{L^{\infty}(\R^2)}\le c\,|t|^{-\frac12}\|\phi\|_{L^1}
\end{equation}
and
\begin{equation} \label{nh_decay.2}
\|D^{\frac{\alpha}{2}-1}_x Q_{ \ge \xi_0}e^{t(\mathcal{L}_{\alpha+1}\mp \partial_x^{-1}\partial_y^2)}\phi\|_{L^{\infty}(\R^2)}\le c\,|t|^{-1}\|\phi\|_{L^1}.
\end{equation}
\end{lemma}

\begin{proof}
The beginning of the proof of \eqref{nh_decay.2} is similar to the one of Lemma \ref{decay}. If one defines 
\begin{equation*}
 D^{\frac{\alpha}{2}-1}_xI_{\alpha+1}(x,y,t)=(2\pi)^{-1} \int_{\R^2} |\xi|^{\frac{\alpha}{2}-1}\psi_0(\xi)e^{it(w_{\alpha+1}(\xi)+\frac{\eta^2}{\xi})} e^{i(x\xi+y\eta)}\,d\xi d\eta \, ,
 \end{equation*}
 one gets arguing as in \eqref{oi2} that
\begin{equation*}
D^{\frac{\alpha}{2}-1}_xI_{\alpha+1}(x,y,t)=\frac{c}{t^{\frac12}}\int_{\R} |\xi|^{\frac{(\alpha-1)}{2}} \psi_0(\xi)  \, e^{i\sgn(\xi)\frac{\pi}{4}}  
e^{i\xi\big(x+\frac{y^2}{4t}\big)} e^{itw_{\alpha+1}(\xi)}\,d\xi \, .
\end{equation*}
Define 
\begin{equation*}
J_{\alpha+1}(\l)= \int_{\R}  \psi_0(\xi)|\xi|^{\frac{(\alpha-1)}{2}} \, e^{i\sgn(\xi)\frac{\pi}{4}}  
\,e^{i\l\xi} e^{itw_{\alpha+1}(\xi)}\,d\xi.
\end{equation*}
To conclude the proof of estimate \eqref{nh_decay.2}, it is then enough to prove that 
\begin{equation*}
\big|J_{\alpha+1}(\l)\big|\le c|t|^{-\frac12} \, .
\end{equation*}
But this is a consequence of Theorem 2.2 in \cite{KPV1} by using the hypotheses on the symbol $w_{\alpha+1}$ in Definition \ref{def_symb}. 

The proof of estimate \eqref{nh_decay.1} is actually easier. If one defines 
\begin{equation*}
 \widetilde{I}_{\alpha+1}(x,y,t)=(2\pi)^{-1} \int_{\R^2} \chi(\xi/\xi_0)e^{it(w_{\alpha+1}(\xi)+\frac{\eta^2}{\xi})} e^{i(x\xi+y\eta)}\,d\xi d\eta \, ,
 \end{equation*}
one gets arguing as in \eqref{oi2} that
\begin{equation*}
\widetilde{I}_{\alpha+1}(x,y,t)=\frac{c}{t^{\frac12}}\int_{\R}  \chi(\xi/\xi_0)  \, e^{i\sgn(\xi)\frac{\pi}{4}}  
e^{i\xi\big(x+\frac{y^2}{4t}\big)} e^{itw_{\alpha+1}(\xi)}\,d\xi \, .
\end{equation*}
Thus, we deduce trivially that 
\begin{equation*}
\big|\widetilde{I}_{\alpha+1}(x,y,t)\big|\lesssim {t^{-\frac12}}\, ,
\end{equation*}
which yields estimate \eqref{nh_decay.1}.
\end{proof}

By using the Stein-Thomas argument, we deduce the Strichartz estimates for solutions of \eqref{linearfKP}.
\begin{proposition}\label{strichartz}
Let $0<\alpha \le 2$. Then, the following  estimates hold
\begin{equation}\label{strichartz1}
\|D^{\frac{1}{q}(\frac{\alpha}{2}-1)}_xU_{\alpha}(t)\phi\|_{L^q_tL^r_{xy}} \le c\,\|\phi\|_{L^2_{xy}}
\end{equation}
and 
\begin{equation}
\|\int_0^t D^{\frac{2}{q}(\frac{\alpha}{2}-1)}_xU_{\alpha}(t-t') F(t')\,dt'\|_{L^q_TL^r_{xy}}\le c\,\|F\|_{L^{q'}_TL^{r'}_{xy}}
\end{equation}
for 
\begin{equation*}
\left\{\begin{array}{l}
2\le r <\infty \\ 2\le q\le \infty
\end{array} \right.
\hskip10pt\text{satisfying}\hskip10pt   \dfrac{1}{r}+\dfrac{1}{q}=\dfrac{1}{2} \hskip10pt\text{and}\hskip10pt \dfrac{1}{q}+\dfrac{1}{q'}=\dfrac{1}{r}+\dfrac{1}{r'}=1 \, .
\end{equation*}
\end{proposition}

\smallskip
Since the endpoint Strichartz estimate corresponding to $(q,r)=(2,+\infty)$ is not known, we need to lose a little bit of regularity in both space directions $x$ and $y$ in order to control this norm.
\begin{corollary}\label{cor-se} Let $\alpha \in (0,2]$. For each $T>0$, $\delta>0$, there exists $\widetilde{\kappa}_{\delta}>0$ such that 
\begin{equation}
\|U_{\alpha}(t)\phi\|_{L^2_TL^{\infty}_{xy}} \le c_{\delta}T^{\widetilde{\kappa}_{\delta}}\big(\|\phi\|_{L^2}+\|D^{\frac{1}{2}(1-\frac{\alpha}{2})+\delta}_x\phi\|_{L^2}+\|D^{\frac12(1-\frac{\alpha}{2})}_xD^{\delta}_y\phi\|_{L^2}\big).
\end{equation}
\end{corollary}

\begin{proof}  Using Sobolev embedding, $\delta >\frac{2}{r}$,  for $r$ sufficiently large satisfying the conditions in Proposition \ref{strichartz}
and applying  \eqref{strichartz1} we have
\begin{equation}
\|U_{\alpha}(t)\phi\|_{L^2_TL^{\infty}_{xy}}\le c_{\delta} T^{\frac{q-2}{2q}}\,\|U_{\alpha}(t) J^{\delta}_{x,y}\phi\|_{L^q_TL^r_{xy}}\le c_{\delta}T^{\widetilde{\kappa}_{\delta}}\| J^{\delta}_{x,y}D_x^{\frac12(1-\frac{\alpha}2)}\phi\|_{L^2}.
\end{equation}
The proof of Corollary \ref{cor-se} follows.
\end{proof}

\begin{remark}  \label{nh_strichartz}
The estimates in Proposition \ref{strichartz} and Corollary \ref{cor-se} also hold if one replaces $U_{\alpha}(t)$ by $Q_{\ge \xi_0}e^{t(\mathcal{L}_{\alpha+1}\mp \partial_x^{-1}\partial_y^2)}$. In order to handle the low frequencies in $x$, we deduce directly from \eqref{nh_decay.1}, the Stein-Thomas argument and the Hardy-Littlewood-Sobolev estimate that 
\begin{equation} \label{nh_strichartz.1}
\|Q_{\le \xi_0}e^{t(\mathcal{L}_{\alpha+1}\mp \partial_x^{-1}\partial_y^2)}\phi\|_{L^4_TL^{\infty}_{xy}} \lesssim \|\phi\|_{L^2_{xy}}.
\end{equation}
\end{remark}

As in \cite{K} the main ingredient in our analysis is a refined Strichartz estimates for solutions of the linear equation
\begin{equation}\label{lp2}
\partial_t u -D^{\alpha}_x\partial_x u+\partial_x^{-1}\partial_y^2u=F.
\end{equation}

More precisely,

\begin{lemma}\label{l3}
Let $\alpha\in(0,2]$, $s_{\alpha}:= \frac32+\frac12(1-\frac{\alpha}{2})$, $0<\delta\le 1$ and $T>0$. Suppose that $w$ is a solution of the linear problem \eqref{lp2}.
Then, there exists $\frac12<\kappa_{\delta}<1$ and $c_{\delta}>0$ such that
\begin{align}\label{refstrichartz}
\|\partial_x w\|_{L^1_TL^{\infty}_{xy}}\le &\;c_{\delta} T^{\kappa_{\delta}} \Big( \underset{[0,T]}{\sup}\| J^{s_{\alpha}+2\delta}_xw\|_{L^2_{xy}}+
\underset{[0,T]}{\sup}\| J^{s_{\alpha}+\delta}_xD^{\delta}_yw\|_{L^2_{xy}} \nonumber \\
& + \int_0^T (\|J^{s_{\alpha}-1+2\delta}_xF(\cdot,t)\|_{L^2_{xy}}+ \|J^{s_{\alpha}-1+\delta}_xD^{\delta}_yF(\cdot,t)\|_{L^2_{xy}})dt \Big).
\end{align}
\end{lemma}

\begin{proof} To establish the estimate \eqref{refstrichartz} we use the argument in \cite{K}. 

First, we use a Littlewood-Paley decomposition of $w$ in the $\xi$ variable.  That is, let $\varphi \in C^{\infty}_0(\frac12<|\xi|<2)$ and $\chi\in C_0^{\infty}(|\xi|<2)$ such that 
$\underset{k=1}{\overset{\infty}{\sum}} \varphi(2^{-k}\xi) +\chi(\xi)=1.$

For $\lg=2^k$, $k\ge 1$, define $w_{\lg}= Q_k w$, where $\widehat{Q_kw}(\xi,\eta)= \varphi(2^{-k}\xi)\widehat{w}(\xi,\eta)$,
$w_0= Q_0 w$, and $\widehat{Q_0w}(\xi,\eta)= \chi(\xi)\widehat{w}(\xi,\eta)$. 

We decompose $\tilde{w}=\underset{k\ge1}{\sum} Q_k(w)$, then $w=\tilde{w}+w_0$. We first estimate $\|\partial_x w_0\|_{L^1_TL^{\infty}_{xy}}$. Noticing that $w_0$ is solution to the integral equation
\begin{equation} \label{se_ie}
\partial_xw_0(t)=U_{\alpha}(t)\partial_xQ_0w(0)+\int_0^tU_{\alpha}(t-t')\partial_xQ_0F(\cdot,t')dt' \, ,
\end{equation}
we deduce combining H\"older's inequality in time, Corollary \ref{cor-se} and Bernstein's inequalities that
\begin{align} \label{rse5}
\|\partial_x w_0\|_{L^1_TL^{\infty}_{xy}} &\lesssim T^{\frac12} \Big(\|U_{\alpha}(t)Q_0w(0)\|_{L^2_TL^{\infty}_{xy}}+\int_0^T\|U_{\alpha}(t-t')Q_0F(\cdot,t')\|_{L^{2}_TL^{\infty}_{xy}}dt'\Big)
\nonumber \\&\lesssim T^{\frac12+\widetilde{\kappa}_{\delta}}\Big(\| J^{\delta}_yw(0)\|_{L^2}+\int_0^T  \|J^{\delta}_yF(\cdot,t)\|_{L^2_{xy}}dt \Big).
\end{align}

Next we estimate $\|\partial_x w_{\Lambda}\|_{L^1_TL^{\infty}_{xy}}$ when $\Lambda=2^{k}$, $k \ge 1$. To do so, we split $[0,T]=\underset{j}{\cup} I_j$, where
$I_j=[a_j, b_j]$ and $b_j-a_j=cT/\lg$, $j=1, \dots, \lg$. 

Since  $\xi\eta( 2^{-k}\xi)$ has Fourier inverse whose $L^1$ norm in $x$ is bounded by $C\lg$ and using the Cauchy-Schwarz
inequality it follows that 
\begin{equation}\label{rs1}
\begin{split}
\|\partial_x w_{\lg}\|_{L^1_TL^{\infty}_{xy}}&\le \underset{j}{\sum} \|\partial_x w_{\lg}\|_{L^1_{I_j}L^{\infty}_{xy}}
\lesssim \lg \underset{j}{\sum} \|w_{\lg}\|_{L^1_{I_j}L^{\infty}_{xy}}\lesssim  (T\lg)^{\frac12} \, \underset{j}{\sum} \|w_{\lg}\|_{L^2_{I_j}L^{\infty}_{xy}} \, .
\end{split}
\end{equation}

Next employing Duhamel's formula, in each $I_j$, we obtain, for $t\in I_j$,
\begin{equation}\label{rs2}
w_{\lg}(t)= U_{\alpha}(t-a_j) w_{\lg}(\cdot, a_j)+\int_{a_j}^t U_{\alpha}(t-t')F_{\lg}(t') dt'. 
\end{equation}

Thus combining \eqref{rs2} with \eqref{rs1}, we deduce that $\|\partial_x w_{\lg}\|_{L^1_TL^{\infty}_{xy}}$ can be bounded by
\begin{displaymath}
c(T\lg)^{\frac12} \underset{j}{\sum} \Big(\|U_{\alpha}(t-a_j) w_{\lg}(a_j)\|_{L^2_{I_j}L^{\infty}_{xy}}
+  \big\|\int_{a_j}^t U_{\alpha}(t-t') F_{\lg}(t')\,dt'\big\|_{L^2_{I_j}L^{\infty}_{xy}}\Big) \, .
\end{displaymath}
Then, it follows from Corollary \ref{cor-se} that
\begin{align}
\|\partial_x w_{\lg}\|_{L^1_TL^{\infty}_{xy}} &
\lesssim \lg^{\frac12}T^{\frac12+\widetilde{\kappa}_{\delta}} \, \underset{j}{\sum} \Big(\| J^{\frac12(1-\frac{\alpha}2)+\delta}_x \,w_{\lg}(a_j)\|_{L^2_{xy}}
+ \|D^{\frac12(1-\frac{\alpha}2)}_xD^{\delta}_y\, w_{\lg}(a_j)\|_{L^2_{xy}}\Big) \nonumber\\
&\hskip10pt + \lg^{\frac12}T^{\kappa_{\delta}}\, \underset{j}{\sum} \int_{I_j}\Big(\|J^{\frac12(1-\frac{\alpha}2)+\delta}_x \, F_{\lg}(t')\|_{L^2_{xy}}+\|D^{\frac12(1-\frac{\alpha}2)}_xD^{\delta}_y \, F_{\lg}(t')\|_{L^2_{xy}}\Big)\,dt \nonumber\\
&\lesssim T^{\kappa_{\delta}}\underset{t\in [0,T]}{\sup}\,\big(\| J^{s_{\alpha}+\delta}_x \,w_{\lg}(t)\|_{L^2_{xy}}+\| J^{s_{\alpha}}_xD^{\delta}_y\,w_{\lg}(t)\|_{L^2_{xy}}\big)\nonumber\\
&\hskip15pt + T^{\kappa_{\delta}}\int_0^T \Big(\|J^{s_{\alpha}-1+\delta}_x \, F_{\lg}(t')\|_{L^2_{xy}}+\|J^{s_{\alpha}-1}_xD^{\delta}_y \, F_{\lg}(t')\|_{L^2_{xy}}\Big)\,dt' \, , \label{rs3}
 \end{align}
 where $\kappa_{\delta} :=\frac12+\widetilde{\kappa}_{\delta} \in (\frac12,1)$.
 
 Thus,
 \begin{equation}\label{rse4}
 \begin{split}
 \|\partial_x \tilde{w}\|_{L^1_TL^{\infty}_{xy}} 
 &\lesssim  T^{\kappa_{\delta}}\underset{k \ge 1}{\sum}\; \Big(\underset{[0,T]}{\sup} \big(\|J^{s_{\alpha}+\delta}_xQ_k(w)\|_{L^2_{xy}}
 +\|J^{s_{\alpha}}_xD_y^{\delta} Q_k(w)\|_{L^2_{xy}})\\
 &\hskip25pt +\int_0^T \big(\|J_x^{s_{\alpha}-1+\delta} Q_k(F)\|_{L^2_{xy}}+\|J_x^{s_{\alpha}-1+\delta}D^{\delta}_y Q_k(F)\|_{L^2_{xy}}\big)\,dt \Big)\\
 &\le  T^{\kappa_{\delta}}\underset{k \ge 1}{\sum}\; 2^{-k\delta}\Big(\sup_{[0,T]}\big( \|J^{s_{\alpha}+2\delta}_xQ_k(w)\|_{L^2_{xy}} +\|J^{s_{\alpha}+\delta}_xD_y^{\delta} Q_k(w)\|_{L^2_{xy}}\big)\\
 &\hskip25pt +\int_0^T\big( \|J^{s_{\alpha}-1+2\delta}_x Q_k(F)\|_{L^2_{xy}}+\|J_x^{s_{\alpha}-1+\delta}D^{\delta}_y Q_k(F)\|_{L^2_{xy}}\big)dt\Big).
 \end{split}
\end{equation}

Therefore estimate \eqref{refstrichartz} follows gathering \eqref{rse5} and \eqref{rse4}.
\end{proof}

A similar refined Strichartz estimate also holds for the linear version of \eqref{nh_fKP}
\begin{equation} \label{nh_lp2}
\partial_tu-\mathcal{L}_{\alpha+1}u \pm\partial_x^{-1}\partial_y^2u=F \, .
\end{equation}
\begin{lemma}\label{nh_l3}
Let $\alpha\in(0,2]$, $s_{\alpha}:= \frac32+\frac12(1-\frac{\alpha}{2})$, $0<\delta\le 1$ and $T>0$. Suppose that $w$ is a solution of the linear problem \eqref{nh_lp2}.
Then, there exists $\frac12<\kappa_{\delta}<1$ and $c_{\delta}>0$ such that
\begin{align}\label{nh_refstrichartz}
\|\partial_x w\|_{L^1_TL^{\infty}_{xy}}\le &\;c_{\delta} \big(T^{\kappa_{\delta}}+T^{\frac34}\big) \Big( \underset{[0,T]}{\sup}\| J^{s_{\alpha}+2\delta}_xw\|_{L^2_{xy}}+
\underset{[0,T]}{\sup}\| J^{s_{\alpha}+\delta}_xD^{\delta}_yw\|_{L^2_{xy}} \nonumber \\
& + \int_0^T (\|J^{s_{\alpha}-1+2\delta}_xF(\cdot,t)\|_{L^2_{xy}}+ \|J^{s_{\alpha}-1+\delta}_xD^{\delta}_yF(\cdot,t)\|_{L^2_{xy}})dt \Big).
\end{align}
\end{lemma}

\begin{proof} The proof is similar to the one of Lemma \ref{l3}. With the same notations, we make the decomposition $w=\widetilde{w}+w_0$, where $$\widetilde{w}=\sum_{k \ge k_0}Q_k(w) \quad  \text{and}  \quad w_0=\big(\chi(\xi/\xi_0)\widehat{w}(\xi,\eta) \big)^{\vee} \, ,$$
and $\xi_0$ is given in Definition \ref{def_symb}. The estimate for $\widetilde{w}$ follows exactly in the same way than the one in Lemma \ref{l3} since the Strichartz estimates are similar.

To handle the low frequency part $w_0$, we apply the Strichartz estimate \eqref{nh_strichartz.1} on the integral equation \eqref{se_ie} to get
\begin{align*}
\|\partial_x w_0\|_{L^1_TL^{\infty}_{xy}} &\lesssim T^{\frac34} \Big(\|U_{\alpha}(t)Q_0w(0)\|_{L^4_TL^{\infty}_{xy}}+\int_0^T\|U_{\alpha}(t-t')Q_0F(\cdot,t')\|_{L^{4}_TL^{\infty}_{xy}}dt'\Big)
\nonumber \\&\lesssim T^{\frac34}\Big(\| w(0)\|_{L^2}+\int_0^T  \|F(\cdot,t)\|_{L^2_{xy}}dt \Big).
\end{align*}
This concludes the proof of Lemma \ref{nh_l3}.
\end{proof}


\subsection{Energy estimates}
Without loss of generality, we will work in the rest of the paper with the equation \eqref{fKP} with the constant $\kappa=-1$. 
The proofs in the case $\kappa=1$ or for the non-homogenous equation \eqref{nh_fKP} follow similarly.

\medskip
Using the definition \eqref{Xnorm}, energy estimates and Kato-Ponce commutators estimates (Lemma \ref{Kato-Ponce}) we obtain the
following {\it a priori} estimate.

\begin{lemma}\label{l01} Let $\alpha \in (0,2]$ and $T>0$.  Let $u \in C([0,T] ; X^{\infty}(\mathbb R^2))$ be a solution of the IVP \eqref{fKP}.  Then, there exists a positive constant $C_0$ such that
\begin{equation} \label{enest}
\|u\|_{L^{\infty}_TX^{\sigma}_{xy}}^2 \le \| u_0\|_{X^{\sigma}}^2+C_0\|\partial_xu\|_{L^1_TL^{\infty}_{xy}}\|u\|_{L^{\infty}_TX^{\sigma}_{xy}}^2 \, ,
\end{equation}
for any $\sigma \in [1,5]$.
\end{lemma}

\begin{proof}
We first deal with the $\|J^s_x(\cdot)\|_{L^2_{xy}}$ part of the $X^s$ norm. Applying $J^s_x$ to the equation in \eqref{fKP}, multiplying by $J^s_xu$ and integrating in space leads to 
\begin{displaymath}
\frac12 \frac{d}{dt} \int_{\mathbb R^2} (J^s_xu)^2 \, dxdy=-\int_{\mathbb R^2} [J^s_x,u]\partial_xu \, J^s_xu \, dxdy-\int_{\mathbb R^2}uJ^s_x\partial_xu J^s_xu dxdy \, .
\end{displaymath}
We use the commutator estimate \eqref{Kato-Ponce.1} and H\"older's inequality in the $y$ variable to deal the first term on the right-hand side and integrate by parts in $x$ and use H\"older's inequality in both $x$ and $y$ to deal with the second term, which implies 
\begin{equation} \label{enest.1}
 \frac{d}{dt}\|J^s_xu\|_{L^2_{xy}}^2 \lesssim \|\partial_xu\|_{L^{\infty}_{xy}} \|J^s_xu\|_{L^2_{xy}}^2 \, .
\end{equation}

To deal with the $\|\partial_x^{-1}\partial_y(\cdot)\|_{L^2_{xy}}$ part of the $X^s$ norm, we apply $\partial_x^{-1}\partial_y$ to the equation in \eqref{fKP}, multiply by $\partial_x^{-1}\partial_yu$, integrate in space and integrate by parts to deduce that 
\begin{displaymath}
\frac12 \frac{d}{dt} \int_{\mathbb R^2} (\partial_x^{-1}\partial_yu)^2 \, dxdy=-\int_{\mathbb R^2}u\partial_yu \partial_x^{-1}\partial_yu dxdy=\frac12\int_{\mathbb R^2}\partial_xu( \partial_x^{-1}\partial_yu)^2 dxdy
\end{displaymath} 
Hence, it follows by applying H\"older's inequality that 
\begin{equation} \label{enest.2}
\frac12 \frac{d}{dt}\|\partial_x^{-1}\partial_yu\|_{L^2_{xy}}^2 \lesssim \|\partial_xu\|_{L^{\infty}_{xy}} \|\partial_x^{-1}\partial_yu\|_{L^2_{xy}}^2 \, .
\end{equation}

Therefore, we conclude the proof of estimate \eqref{enest} gathering \eqref{enest.1} and \eqref{enest.2} and integrating in time. 
\end{proof}

\subsection{Estimates for the Strichartz norm}\label{nonlinest}
First, we derive an \textit{a priori} estimate for the norm $\|\partial_xu\|_{L^1_TL^{\infty}_{xy}}$ based on the refined Strichartz estimate derived in Lemma \ref{l3}. Note that we also need to control $\|u\|_{L^1_TL^{\infty}_{xy}}$ in the argument. This is the main result of this section. 
\begin{lemma} \label{apriorilemma} Let $\alpha \in (0,2]$, $s_{\alpha}=2-\frac{\alpha}4$ and $T>0$.  Let $u \in C([0,T] ; X^{\infty}(\mathbb R^2))$ be a solution of the IVP \eqref{fKP}.  Then, for any $s>s_{\alpha}$, there exist $\kappa_s \in (\frac12,1)$ and $C_s>0$ such that 
\begin{equation}\label{fT}
f(T):=\int_0^T \|\partial_xu(t)\|_{L^{\infty}_{xy}}\,dt + \int_0^T \|u(t)\|_{L^{\infty}_{xy}}\,dt 
\end{equation}
satisfies 
\begin{equation}\label{apriori}
f(T) \le C_{s}T^{\kappa_s}(1+f(T))\|u\|_{L^{\infty}_TX^s} \, .
\end{equation}
\end{lemma}

\begin{proof}
Let us fix a constant $\delta_0$ such that $0<\delta_0<s-s_{\alpha}$.

We first estimate $\|\partial_xu\|_{L^1_TL^{\infty}_{xy}}$. Writing $F=-u\partial_xu$ in equation \eqref{lp2} and using Lemma \ref{l3} we get
\begin{align}\label{l1-derivative}
\|\partial_x u\|_{L^1_TL^{\infty}_{xy}}&\le \;c_{\delta}T^{\kappa_{\delta}} \Big( \underset{(0,T)}{\sup}\| J^{s_{\alpha}+2\delta}_xu\|_{L^2_{xy}}+
\underset{(0,T)}{\sup}\| J^{s_{\alpha}+\delta}_xD^{\delta}_y u\|_{L^2_{xy}}\Big)\\
&\hskip10pt + c_{\delta} T^{\kappa_{\delta}} \int_0^T \|J^{s_{\alpha}-1+2\delta}_x(u\partial_xu)\|_{L^2_{xy}}\,dt 
 \\ & \quad + c_{\delta} T^{\kappa_{\delta}}\int_0^T \|J^{s_{\alpha}-1+\delta}_xD^{\delta}_y(u\partial_xu)\|_{L^2_{xy}}\,dt , \nonumber
\end{align}
where $0<\delta<\delta_0$ will be determined during the proof.

We will bound each of the terms on the right hand side of \eqref{l1-derivative}.

First, it is clear by choosing $0<\delta<\delta_0/2$ that
\begin{equation}\label{RHS1}
\underset{(0,T)}{\sup}\| J^{s_{\alpha}+2\delta}_xu\|_{L^2_{xy}}\le \|u\|_{L^{\infty}_TX^s} \, .
\end{equation}

To bound $\underset{(0,T)}{\sup}\| J^{s_{\alpha}+\delta}_xD^{\delta}_y u\|_{L^2_{xy}}$ we follow the same argument used in \cite{K}. More precisely,
we use Young's inequality to obtain
\begin{equation}\label{rhs2a}
(1+|\xi|)^{s_{\alpha}+\delta}\cdot |\eta|^{\delta}\le  (1+|\xi|)^{s_{\alpha}+2\delta}\Big(\frac{|\eta|}{|\xi|}\Big)^{\delta}
\lesssim (1+|\xi|)^{\frac{s_{\alpha}+2\delta}{1-\delta}}+\frac{|\eta|}{|\xi|}.
\end{equation}
Thus choosing $\delta$ sufficiently small such that $(s_{\alpha}+2\delta)/(1-\delta)\le s_{\alpha}+\delta_0$ we have after applying Plancherel's identity
that
\begin{equation}\label{RHS2}
\underset{(0,T)}{\sup}\| J^{s_{\alpha}+\delta}_xD^{\delta}_y u\|_{L^2_{xy}}\lesssim 
\|J^s_xu\|_{L^{\infty}_TL^2_{xy}}+\|\partial_x^{-1}\partial_yu\|_{L^{\infty}_TL^2_{xy}} 
\lesssim \|u\|_{L^{\infty}_TX^s} .
\end{equation}

Next we will bound $\;\int_0^T \|J^{s_{\alpha}-1+2\delta}(u\partial_xu)\|_{L^2_{xy}}\,dt\;$. To do so we observe that
\begin{align}\label{rhs3a}
\int_0^T \|J^{s_{\alpha}-1+2\delta}_x(u\partial_xu)\|_{L^2_{xy}}\,dt &\le \int_0^T \| u\partial_xu\|_{L^2_{xy}}\,dt
+\int_0^T \|D^{s_{\alpha}-1+2\delta}_x(u\partial_xu)\|_{L^2_{xy}}\,dt \nonumber \\
&=: I_1+I_2.
\end{align}
Then, H\"older's inequality give us
\begin{equation}\label{rhs3b}
I_1 \le \underset{[0,T]}{\sup} \|u(t)\|_{L^2_{xy}}  \int_0^T \|\partial_xu(t)\|_{L^{\infty}_{xy}}\,dt \le f(T)\|u\|_{L^{\infty}_TX^s} .
\end{equation}

Since $s_{\alpha}-1+2\delta\in (0,1)$, to bound $I_2$ we employ the fractional Leibniz rule \eqref{Leibniz-1d} to obtain
\begin{equation}\label{rhs3c}
\begin{split}
I_2&\lesssim \int_0^T (\|D^{s_{\alpha}-1+2\delta}_x\partial_xu\|_{L^2_{xy}}\|u\|_{L^{\infty}_{xy}}
+\|\partial_xu\|_{L^{\infty}_{xy}}\|D^{s_{\alpha}-1+2\delta}_xu\|_{L^2_{xy}})\,dt\\
&\lesssim  f(T)\|u\|_{L^{\infty}_TX^s} \, .
\end{split}
\end{equation}

Thus gathering together the estimates \eqref{rhs3a}, \eqref{rhs3b} and \eqref{rhs3c} yields
\begin{equation}\label{RHS3}
\int_0^T \|J^{s_{\alpha}-1+2\delta}_x  (u\partial_xu)(t)\|_{L^2_{xy}}\,dt \lesssim  f(T)\|u\|_{L^{\infty}_TX^s} \, .
\end{equation}

Finally we estimate the last term in \eqref{l1-derivative}. Hence
\begin{equation}\label{rhs4a}
\begin{split}
\int_0^T \|J^{s_{\alpha}-1+\delta}_xD^{\delta}_y(u\partial_xu)\|_{L^2_{xy}}\,dt &\le \int_0^T \| D^{\delta}_y(u\partial_xu)\|_{L^2_{xy}}\,dt\\ & \quad
+\int_0^T \|D^{s_{\alpha}-1+2\delta}_xD^{\delta}_y(u\partial_xu)\|_{L^2_{xy}}\,dt\\
&=: J_1+J_2 \, .
\end{split}
\end{equation}

We use the fractional Leibniz rule \eqref{Leibniz-1d} in the $y$ variable and H\"older's inequality in the $x$ variable to obtain
\begin{equation}\label{rhs4b}
\begin{split}
&J_1\lesssim \int_0^T (\|D^{\delta}_y\partial_xu\|_{L^2_{xy}}\|u\|_{L^{\infty}_{xy}}
+\|\partial_xu\|_{L^{\infty}_{xy}}\|D^{\delta}_yu\|_{L^2_{xy}})\,dt \, .
\end{split}
\end{equation}

An argument similar to the one given in \eqref{rhs2a} gives us that
\begin{equation}\label{rhs4ba}
|\eta|^{\delta}=|\xi|^{\delta}\Big(\frac{|\eta|}{|\xi|}\Big)^{\delta}\le (1+|\xi|)^{\delta}\Big(\frac{|\eta|}{|\xi|}\Big)^{\delta}
\lesssim (1+|\xi|)^{\frac{\delta}{1-\delta}}+\frac{|\eta|}{|\xi|},
\end{equation}
where $\delta$ is choosen so that $\;\frac{\delta}{1-\delta} \le s_{\alpha}\;$ and
\begin{equation}\label{rhs4bb}
|\xi| |\eta|^{\delta}=|\xi|^{\delta+1}\Big(\frac{|\eta|}{|\xi|}\Big)^{\delta}\le (1+ |\xi|)^{\delta+1}\Big(\frac{|\eta|}{|\xi|}\Big)^{\delta}
\lesssim (1+|\xi|)^{\frac{1+\delta}{1-\delta}}+\frac{|\eta|}{|\xi|},
\end{equation}
where $\delta$ is selected so that  $\frac{1+\delta}{(1-\delta)}\le s_{\alpha}$. Plancherel's identity implies then that
\begin{equation}\label{rhs4c}
J_1  \lesssim  f(T)\|u\|_{L^{\infty}_TX^s} \, .
\end{equation}

Next we use estimate \eqref{leibniz-2d} to obtain
\begin{equation}\label{lte1}
\begin{split}
J_2&=\int_0^T \|D^{s_{\alpha}-1+\delta}_xD^{\delta}_y(u\partial_xu)\|_{L^2_{xy}}\,dt\\
&\lesssim \int_0^T \|\partial_x u\|_{L^{\infty}_{xy}}\|D^{s_{\alpha}-1+\delta}_xD^{\delta}_yu\|_{L^2_{xy}}
+ \int_0^T \|\partial_xD^{s_{\alpha}-1+\delta}_xD^{\delta}_y u\|_{L^2_{xy}}\|u\|_{L^{\infty}_{xy}}\\
&\hskip10pt+\int_0^T \|D^{s_{\alpha}-1+\delta}_xu\|_{L^{\infty}_{xy}}\|D^{\delta}_y\partial_x u\|_{L^2_{xy}}
+ \int_0^T \|D^{\delta}_yu\|_{L^{p_1}_{xy}}\|D^{s_{\alpha}-1+\delta}_x\partial_xu\|_{L^{q_1}_{xy}}\\
&= J_{21}+J_{22}+J_{23}+J_{24} ,
\end{split}
\end{equation}
where $1<p_1,q_1<\infty$ satisfy $\frac1{p_1}+\frac1{q_1}=\frac12$.

Noticing that 
\begin{equation}
|\xi|^{s_{\alpha}-1+\delta}|\eta|^{\delta}\le (1+|\xi|)^{s_{\alpha}-1+2\delta}\Big(\frac{|\eta|}{|\xi|}\Big)^{\delta} \lesssim (1+|\xi|)^{\frac{s_{\alpha}-1+\delta}{1-\delta}}+ \frac{|\eta|}{|\xi|} \, ,
\end{equation}
it follows by choosing $\delta$ such that $\;\dfrac{s_{\alpha}-1+\delta}{1-\delta}\le s_{\alpha}$  that
\begin{equation}\label{j21}
J_{21}  \lesssim f(T)\|u\|_{L^{\infty}_TX^s} \, .
\end{equation}

To bound $J_{22}$ we proceed as before. Observing that 
\begin{equation}
\begin{split}
|\xi| |\xi|^{s_{\alpha}-1+\delta}|\eta|^{\delta}=|\xi| |\xi|^{s_{\alpha}-1+2\delta}\Big(\frac{|\eta|}{|\xi|}\Big)^{\delta}&\lesssim
(1+|\xi|)^{\frac{s_{\alpha}+2\delta}{1-\delta}}+ \frac{|\eta|}{|\xi|} \, ,
\end{split}
\end{equation}
and choosing $\delta$ such that $\;\dfrac{s_{\alpha}+2\delta}{(1-\delta)}\le s_{\alpha}+\delta_0$  yield
\begin{equation}\label{j22}
J_{22} \lesssim f(T)\|u\|_{L^{\infty}_TX^s} \, .
\end{equation}

Next we estimate $J_{23}$. We will use the estimate \eqref{kenig-(2.5)} in Lemma \ref{interpolated} to obtain
\begin{equation}
\begin{split}
J_{23}&\lesssim \int_0^T \|D^{s_{\alpha}-1+\delta}_x u\|_{L^{\infty}_{xy}}\|D^{\delta}_y\partial_xu\|_{L^2_{xy}}\\
&\lesssim \underset{[0,T]}{\sup}\|D^{\delta}_y\partial_xu\|_{L^2_{xy}} 
\int_0^T \Big(\|u\|_{L^{\infty}_{xy}}+\|\partial_x u\|_{L^{\infty}_{xy}}\Big)\,dt.
\end{split}
\end{equation}
On the other hand, noticing that
\begin{equation*}
|\xi||\eta|^{\delta}\lesssim (1+|\xi|)^{1+\delta}\,\Big(\frac{|\eta|}{|\xi|}\Big)^{\delta}
\lesssim (1+|\xi|)^{\frac{1+\delta}{1-\delta}}+\frac{|\eta|}{|\xi|}
\end{equation*}
whenever $\;\dfrac{1+\delta}{1-\delta}<s_{\alpha}$, Plancherel's identity and Lemma \ref{l01} yield that
\begin{equation}\label{j23}
J_{23}\lesssim f(T)\|u\|_{L^{\infty}_T X^s} \, .
\end{equation}

Finally, to bound $J_{24}$ we employ the inequalities \eqref{kenig-(2.6)} and \eqref{kenig-(2.7)} in Lemma \ref{interpolated}. Hence
we have that
\begin{equation}
\begin{split}
J_{24} &\le \|D^{s_{\alpha}-1+\delta}_x\partial_xu\|_{L^{s_1}_TL^{q_1}_{xy}} \|D^{\delta}_y u\|_{L^{r_1}_TL^{p_1}_{xy}}\\
&\lesssim  \|\partial_xu\|_{L^1_TL^{\infty}_{xy}}^{\theta} \|J^{s_{\alpha}+\delta_0}_x u\|_{L^{\infty}_TL^2_{xy}}^{1-\theta}
 \|u\|_{L^1_TL^{\infty}_{xy}}^{1-\theta}
 \Big( \|D^{\frac12}_yu\|_{L^{\infty}_TL^2_{xy}}+
  \|u\|_{L^{\infty}_TL^2_{xy}}\Big)^{\theta}
 \end{split}
 \end{equation}
 if $\delta>0$ is chosen small enough.
 Since
 \begin{equation*}
 |\eta|^{\frac12}\le \Big(\frac{|\eta|}{|\xi|}\Big)^{\frac12}(1+|\xi|)^{\frac12}\lesssim \frac{|\eta|}{|\xi|}+(1+|\xi|)\lesssim  \frac{|\eta|}{|\xi|}+(1+|\xi|)^{s_{\alpha}},
 \end{equation*}
 with $\delta_0$ as above,  Plancherel's identity yields
 \begin{equation*}
  \|D^{\frac12}_yu\|_{L^{\infty}_TL^2_{xy}}+ \|u\|_{L^{\infty}_TL^2_{xy}}\lesssim  \|u\|_{L^{\infty}_TX^s} \, .
 \end{equation*}
  Therefore,  it follows that
 \begin{equation}\label{j24}
 J_{24}\lesssim  \|\partial_xu\|_{L^1_TL^{\infty}_{xy}}^{\theta}
 \|u\|_{L^1_TL^{\infty}_{xy}}^{1-\theta}\|u\|_{L^{\infty}_TX^s} \lesssim f(T)\|u\|_{L^{\infty}_TX^s}  \, .
 \end{equation}

Gathering together the information in \eqref{RHS1}, \eqref{RHS2}, \eqref{RHS3}, \eqref{rhs4c}, \eqref{j21}, \eqref{j22},
\eqref{j23} and \eqref{j24}  we obtain
\begin{equation} \label{j25}
\|\partial_x u\|_{L^1_TL^{\infty}_{xy}} \le c_sT^{\kappa_s}(1+f(T))\|u\|_{L^{\infty}_TX^s} \, ,
\end{equation}
with $\kappa_s=\kappa_{\delta} \in (\frac12,1)$.

To estimate  $\|u\|_{L^1_TL^{\infty}_{xy}}$, observe by Duhamel's principle that $u$ is solution to the integral equation 
\begin{displaymath}
u(t)=U(t)u_0+\int_0^tU_{\alpha}(t-t')(u\partial_xu) dt' \, .
\end{displaymath}
Then it follows by using H\"older's inequality in time and Corollary \ref{cor-se} that 
\begin{displaymath}
\begin{split}
\|u\|_{L^1_TL^{\infty}_{xy}} &\lesssim T^{\kappa_s}\left(\|J_x^{\frac12(1-\frac{\alpha}2)+\delta}u_0\|_{L^2_{xy}}+\|D_x^{\frac12(1-\frac{\alpha}2)}D_y^{\delta}u_0\|_{L^2_{xy}} \right) \\
& \quad +T^{\kappa_s}\int_0^T\|J_x^{\frac12(1-\frac{\alpha}2)+\delta}(u\partial_xu)\|_{L^2_{xy}}dt\\ &\quad+T^{\kappa_s}\int_0^T\|D_x^{\frac12(1-\frac{\alpha}2)}D_y^{\delta}(u\partial_xu)\|_{L^2_{xy}}dt  \, .
\end{split}
\end{displaymath}
Noting that $\frac12(1-\frac{\alpha}2)+1<2-\frac{\alpha}4=s_{\alpha}$, it can be deduced arguing in the same lines as above (it is actually easier) that 
\begin{equation} \label{j26}
\|u\|_{L^1_TL^{\infty}_{xy}} \le c_sT^{\kappa_s}(1+f(T))\|u\|_{L^{\infty}_TX^s} \, ,
\end{equation}

Therefore, we conclude the proof of Lemma \ref{apriorilemma} combining estimates \eqref{j25} and \eqref{j26}.
\end{proof}

We also will need an \textit{a priori} estimate for $\|\partial_x^2u\|_{L^1_TL^{\infty}_{xy}}$ in order to run the Bona-Smith argument in Section \ref{section_existence}.

\begin{lemma} \label{apriori_lemma2} Let $\alpha \in (0,2]$, $s_{\alpha}=2-\frac{\alpha}4$ and  $T>0$.  Let $u \in C([0,T] ; X^{\infty}(\mathbb R^2))$ be a solution of the IVP \eqref{fKP}.  Then, for any $s>s_{\alpha}$, there exist $\kappa_s \in (\frac12,1)$ and $C_s>0$ such that
\begin{equation}\label{apriori_lemma2.1}
\|\partial_x^2u\|_{L^1_TL^{\infty}_{xy}} \le C_{s}T^{\kappa_s}(1+f(T))\|u\|_{L^{\infty}_TX^{s+1}_{xy}}
+C_{s}T^{\kappa_s}\|\partial_x^2u\|_{L^1_TL^{\infty}_{xy}}\|u\|_{L^{\infty}_TX^s} \, .
\end{equation}
where $f(T)$ is defined in \eqref{fT}.
\end{lemma}

\begin{proof}
The proof of Lemma \ref{apriori_lemma2} follows the lines as the one of Lemma \ref{apriorilemma}. For the sake of brevity, we will omit the details.
\end{proof}

\section{Proof of Theorem \ref{maintheo}}

First, we state a useful anisotropic Sobolev embedding that will be proved in the appendix. 
\begin{lemma} \label{anisSobo}
Let $s>4$. Then, it holds that 
\begin{equation} \label{anisSobo.1}
\|\partial_xu\|_{L^{\infty}_{xy}} \lesssim \|u\|_{X^s} \, ,
\end{equation}
for any $u \in X^s$.
\end{lemma}

Using this embedding and, for instance, the argument in \cite{IN},  we can establish the following local well-posedness result for the IVP \eqref{fKP}.
\begin{lemma}\label{l1}
Let $\alpha \in (0,2]$ and $\sigma \ge 5$. For any $u_0 \in X^{\sigma}(\mathbb R^2)$, there exist a positive time $T=T(\|u_0\|_{X^{5}})>0$ and a unique solution $u \in C([0,T] : X^{\sigma}(\mathbb R^2))$ to \eqref{fKP}. Moreover, for any $R>0$, the map $u_0 \mapsto u$ is continuous from the ball of $X^{\sigma}(\mathbb R^2)$ of radius $R$ containing $u_0$ into $C([0,T(R) : X^{\sigma}(\mathbb R^2)) .$
\end{lemma}

\subsection{\textit{A priori} estimates} \label{section_apriori}
Let $u_0 \in X^{\infty}(\mathbb R^2))$. From the above result, there exists a solution $u \in C([0,T^{\star}) : X^{\infty}(\mathbb R^2))$ to \eqref{fKP}, where $T^{\star}$ is the maximal time of existence of $u$ satisfying $T^{\star} \ge T(\|u_0\|_{X^5})$ and we have the blow-up alternative
\begin{equation} \label{blow-up_alt}
\lim_{t \nearrow  T^{\star}}\|u(t)\|_{X^5_{xy}}=+\infty \quad \text{if} \quad T^{\star}<+\infty \, .
\end{equation}

The following \textit{a priori} estimate holds true.
\begin{lemma} \label{a priori}
Let $\alpha \in (0,2]$, $s_{\alpha}=2-\frac{\alpha}4$ and $s \in (s_{\alpha},5)$. There exist $K_0>0$, $A_s>0$ and $\kappa_0 \in (\frac12,1)$ such that $T^{\star} > (A_s\|u_0\|_{X^s}+1)^{-2}$,
\begin{equation} \label{aprioest}
\|u\|_{L^{\infty}_TX^s} \le 2\|u_0\|_{X^s} \hskip7pt \text{and} \hskip7pt  f(T) \le K_0  \hskip7pt  \text{with}  \hskip7pt  T=(A_s\|u_0\|_{X^s}+1)^{-2}\, .
\end{equation}
We recall that 
\begin{displaymath}
f(T)=\|u\|_{L^1_TL^{\infty}_{xy}}+\|\partial_xu\|_{L^1_TL^{\infty}_{xy}} \, .
\end{displaymath}
\end{lemma}

\begin{proof} For $s_{\alpha}<s \le 5$, let us define
\begin{displaymath}
T_0:=\sup\Big\{ T \in (0,T^{\star}) \ : \ \|u\|_{L^{\infty}_TX^s}^2 \le 4 \|u_0\|_{X^s}^2 \Big\} \, .
\end{displaymath}
Note that the above set is nonempty since $u \in C([0,T^{\star}) : X^{\infty}(\mathbb R^2))$, so that $T_0$ is well-defined. We argue by contradiction assuming that $0<T_0<(A_s\|u_0\|_{X^s}+1)^{-2}$.

By continuity, we have that $\|u\|_{L^{\infty}_{T_0}X^s}^2 \le 4 \|u_0\|_{X^s}^2$. Then estimate \eqref{apriori} yields 
\begin{displaymath}
f(T_0) \le 2C_sT_0^{\frac12}\|u_0\|_{X^s}(1+f(T_0)) \, .
\end{displaymath}
Thus if we fix $A_s=8(1+C_0)C_s$ (where $C_0$ and $C_s$ are respectively defined in  Lemmas \ref{l01} and \ref{apriorilemma}), it follows that 
\begin{displaymath}
f(T_0) \le \frac1{3C_0} \, .
\end{displaymath}
We deduce by using the energy estimate \eqref{enest} with $\sigma=5$ that
\begin{displaymath}
\|u\|_{L^{\infty}_{T_0}X^5_{xy}}^2 \le \frac32 \|u_0\|_{X^{5}}^2 \, .
\end{displaymath}
This implies in view of the blow-up alternative \eqref{blow-up_alt} that $T_0<T^{\star}$.

Now, the energy estimate \eqref{enest} with $\sigma=s$ yields $\|u\|_{L^{\infty}_{T_0}X^s}^2 \le \frac32 \|u_0\|_{X^s}^2$, so that by continuity, $\|u\|_{L^{\infty}_TX^s}^2 \le 3\|u_0\|_{X^s}^2$ for some $T_0<T<T^{\star}$. This contradicts the definition of $T_0$. Then, we argue as above to get the bound for $f(T)$, which concludes the proof of Lemma \ref{a priori}.

\end{proof}

\subsection{Uniqueness and $L^2$-Lipschitz bound of the flow} \label{section_uniq}
Let $u_1$ and $u_2$ be two solutions of the equation in \eqref{fKP} in the class \eqref{maintheo.1} for some positive $T$ with respective initial data $u_1(\cdot,0)=\varphi_1$ and $u_2(\cdot,0)=\varphi_2$. We define the positive number $K$ by 
\begin{equation} \label{K}
K=\max\big\{\|\partial_xu_1\|_{L^1_TL^{\infty}_{xy}} , \|\partial_xu_2\|_{L^1_TL^{\infty}_{xy}} \big\} \, .
\end{equation}
We set $v=u_1-u_2$. Then $v$ satisfies 
\begin{equation} \label{fKPdiff}
\partial_tv-D^{\alpha}_x\partial_xv-\partial_x^{-1}\partial_yv+\frac12\partial_x\big((u_1+u_2)v\big)=0 \, ,
\end{equation}
with initial datum $v(\cdot,0)=\varphi_1-\varphi_2$.

We want to estimate $v$ in $X^0$. First, we deal with the $L^2$ component of $X^0$. We multiply \eqref{fKPdiff} by $v$, integrate in space and integrate by parts to deduce that
\begin{displaymath}
\frac12\frac{d}{dt}\int_{\mathbb R^2} v^2 \, dxdy=-\frac12 \int_{\mathbb R^2} \partial_x\big((u_1+u_2)v\big) \, dxdy=
-\frac14\int_{\mathbb R^2}\partial_x(u_1+u_2)v^2\, dxdy \, .
\end{displaymath}
This implies from H\"older's inequality that 
\begin{displaymath}
\frac12\frac{d}{dt}\|v\|_{L^2_{xy}}^2 \lesssim \big(\|\partial_xu_1\|_{L^{\infty}_{xy}}+ \|\partial_xu_2\|_{L^{\infty}_{xy}}\big)\|v\|_{L^2_{xy}}^2 \, .
\end{displaymath}
Therefore, it follows from Gronwall's inequality that 
\begin{equation} \label{diffL2}
\sup_{t \in [0,T]}\|v(\cdot,t)\|_{L^2_{xy}}=\sup_{t \in [0,T]}\|u_1(\cdot,t)-u_2(\cdot,t)\|_{L^2_{xy}} \le e^{cK}\|\varphi_1-\varphi_2\|_{L^2} \, .
\end{equation}
Estimate \eqref{diffL2} provides the uniqueness result in Theorem \ref{maintheo} by choosing $\varphi_1=\varphi_2=u_0$. 

\subsection{Existence} \label{section_existence}
We will consider the most difficult case where $s_{\alpha}<s <2$. Fix an initial datum $u_0 \in X^s$. 
\smallskip

We will use the Bona-Smith argument \cite{BS}. We regularize the initial datum as follows.
For $\rho\in C^{\infty}_0(\R)$, $\rho(\xi)\in [0,1]$, $\rho(\xi)=0$ for $|\xi|>1$ and $\rho(\xi)=1$ for $|\xi|\le 1/2$, define
\begin{equation*}
	u_{0,n}= P_{\le n}u_0:=P_{\le n}u_0=\{\rho({\xi}/{n}) \widehat{u}_0(\xi, \eta)\}^{\vee} \, ,
\end{equation*}
for any $n \in \mathbb N$, $n \ge 1$.
\smallskip

First, we state some properties on the regularized initial data.
\begin{lemma} \label{BSreg}
Let $\sigma \ge 0$ and $n \ge 1$. Then,
\begin{equation} \label{BSreg.1}
\|P_{\le n}J^{s+\sigma}_x u_0\|_{L^2}\le n^{\sigma} \|J^s_xu_0\|_{L^2} \, , 
\end{equation}
\begin{equation} \label{BSreg.2}
\| P_{\le n}\p^{-1}_x\p_y u_0\|_{L^2}\le  \|\p^{-1}_x\p_yu_0\|_{L^2}
\end{equation}
and 
\begin{equation} \label{BSreg.3}
\| P_{\le n}\p_y u_0\|_{L^2}\le  n\|\p^{-1}_x\p_yu_0\|_{L^2} \, .
\end{equation}
\smallskip

Let $0 \le \sigma \le s$ and $m\ge n\ge 1$. Then, 
\begin{equation} \label{BSreg.4}
\|J^{s-\sigma}_x (u_{0,n}-u_{0,m})\|_{L^2} \underset{n \to +\infty}{=}o(n^{-\sigma}) 
\end{equation}
and
\begin{equation} \label{BSreg.5}
\|\p^{-1}_x\p_y (u_{0,n}-u_{0,m})\|_{L^2}  \underset{n \to +\infty}{=}o(1) \, .
\end{equation}
\end{lemma}

\begin{proof} Estimates \eqref{BSreg.1}, \eqref{BSreg.2} and \eqref{BSreg.3} are easily verified by using Plancherel's identity. For example, we explained how to obtain \eqref{BSreg.3}. Indeed,
\begin{displaymath}
\begin{split}
\| P_{\le n}\partial_y u_0\|_{L^2} & \le  \left(\int_{\eta}\int_{\xi \le n} \eta^2 |\widehat{u}_0(\xi,\eta)|^2 \, d\xi d\eta \right)^{\frac12} \\ &
=  \left(\int_{\eta}\int_{\xi \le n} \Big(\frac{\eta}{\xi}\Big)^2 \xi^2 |\widehat{u}_0(\xi,\eta)|^2 \, d\xi d\eta \right)^{\frac12}
\le n \| \partial_x^{-1}\partial_y u_0\|_{L^2} \, .
\end{split}
\end{displaymath}

Now, we show \eqref{BSreg.4} and \eqref{BSreg.5}. Observe that
\begin{displaymath} 
\|J^{s-\sigma}_x (u_{0,n}-u_{0,m})\|_{L^2}\le \left(\int_{\eta}\int_{|\xi|\ge \frac{n}2} \langle \xi\rangle^{2(s-\sigma)}| \widehat{u}_0|^2\,d\xi d\eta\right)^{\frac12} \underset{n \to +\infty}{=}o(n^{-\sigma}) 
\end{displaymath}
and
\begin{displaymath}
\|\p^{-1}_x\p_y (u_{0,n}-u_{0,m})\|_{L^2}\le \Big\{\int_{\eta}\int_{|\xi|\ge \frac{n}2} |\eta|^2|\xi|^{-2}|\widehat{u}_0|^2\,d\xi d\eta\Big\}^{\frac12} \underset{n \to +\infty}{\to} 0 \, ,
\end{displaymath}
by using the Lebesgue dominated convergence theorem, since $u_0 \in X^s$.
\end{proof}

Now, for each $n \in \mathbb N$, $n \ge 1$, we consider the solution $u_n$ emanating from $u_{0,n}$. In other words, $u_n$ is a solution to the Cauchy problem   
\begin{equation}\label{fKPun}
\begin{cases}
\partial_t u_n -D^{\alpha}_x\partial_x u_n-\kappa \partial_x^{-1}\partial_y^2u_n+ u_n\partial_x u_n=0,\;\;\;(x,y)\in\R^2, t>0, \\
u_n(x,y,0)=u_{0,n}(x,y)=P_{\le n}u_0 \, .
\end{cases}
\end{equation}
From Lemma \ref{a priori}, there exists a positive time 
\begin{equation} \label{defT}
T=(A_s\|u_0\|_{X^s}+1)^{-2} \, ,
\end{equation} 
(where $A_s$ is a positive constant), independent of $n$, such that $u_n \in C([0,T] : X^{\infty}(\mathbb R^2))$ is defined on the time interval $[0,T]$ and satisfies 
\begin{equation} \label{existence.1}
\|u_n\|_{L^{\infty}_TX^s} \le 2\|u_0\|_{X^s} 
\end{equation}
and 
\begin{equation} \label{existence.2}
K:=\sup_{n \ge 1}\big\{ \|u_n\|_{L^1_TL^{\infty}_{xy}}+ \|\partial_xu_n\|_{L^1_TL^{\infty}_{xy}} \big\} <+\infty\, .
\end{equation}
Moreover, we deduce combining \eqref{apriori_lemma2.1}, \eqref{existence.1} and \eqref{existence.2} and taking $A_s$ large enough that
\begin{equation} \label{existence.200}
\|\partial_x^2u_n\|_{L^1_TL^{\infty}_{xy}} \lesssim (1+K)\|u_n\|_{L^{\infty}_TX^{s+1}_{xy}} \, .
\end{equation}

Let $m \ge n \ge 1$. We set $v_{n,m} := u_n-u_m$. Then, $v_{n,m}$ satisfies
\begin{equation} \label{fKPdiffn}
\partial_tv_{n,m}-D^{\alpha}_x\partial_xv_{n,m}-\partial_x^{-1}\partial_yv_{n,m}+\frac12\partial_x\big((u_n+u_m)v_{n,m}\big)=0 \, ,
\end{equation}
with initial datum $v_{n,m}(\cdot,0)=u_{0,n}-u_{0,m}$. We will prove that $\{v_{n,m} \}$ is a Cauchy sequence in $X^s(\mathbb R^2)$. 

Arguing as in Subsection \ref{section_uniq}, we see from Gronwall's inequality and \eqref{BSreg.4} with $\sigma=s$ that
\begin{equation} \label{existence.3}
\| v_{n,m} \|_{L^{\infty}_TL^2_{xy}} \le e^{cK}\| u_{0,n}-u_{0,m} \|_{L^2} \underset{n \to +\infty}{=} o(n^{-s})
\end{equation}
which implies interpolating with \eqref{existence.1} that 
\begin{equation} \label{existence.4} 
\| J_x^{\sigma}v_{n,m} \|_{L^{\infty}_TL^2_{xy}} \le \|J^s_xv_{n,m}\|_{L^{\infty}_TL^2_{xy}}^{\frac{\sigma}s} \|v_{n,m}\|_{L^{\infty}_TL^2_{xy}}^{1-\frac{\sigma}s}\underset{n \to +\infty}{=}o(n^{-(s-\sigma)}) \, ,
\end{equation}
for all $0 \le \sigma <s$. 
\smallskip

Therefore, in order to conclude  that $\{u_n \}$ is a Cauchy sequence in $L^{\infty}([0,T] : X^s(\mathbb R^2)$, it remains to show that $\| J_x^sv_{n,m} \|_{L^{\infty}_TL^2_{xy}}$ and $\| \partial_x^{-1}\partial_y v_{n,m} \|_{L^{\infty}_TL^2_{xy}}$ tend to $0$ as $n$ tends to $+\infty$. This will be done in the next proposition. 

\begin{proposition} \label{prop_exist}
Let $n, \, m \in \mathbb N$ be such that $m \ge n \ge 1$. Then, 
\begin{equation} \label{prop_exist.1}
\| \partial_x^{-1}\partial_y v_{n,m} \|_{L^{\infty}_TL^2_{xy}} \underset{n \to +\infty}{\longrightarrow} 0
\end{equation}
and 
\begin{equation} \label{prop_exist.2}
\| J_x^sv_{n,m} \|_{L^{\infty}_TL^2_{xy}} \underset{n \to +\infty}{\longrightarrow} 0 \, .
\end{equation}
\end{proposition}

In this direction, we first derive a control of $v_{n,m}$ in $L^1([0,T] : W_x^{1,\infty}(\mathbb R^2))$. 
\begin{lemma} \label{lem_exist}
Let $n, \, m \in \mathbb N$ be such that $m \ge n \ge 1$. Then, 
\begin{equation} \label{lem_exist.1}
\|v_{n,m}\|_{L^1_TL^{\infty}_{xy}} \underset{n \to +\infty}{=} o(n^{-1}) +O(n^{-1}\|\partial_x^{-1}\partial_yv_{n,m}\|_{L^{\infty}_TL^2_{xy}})
\end{equation}
and 
\begin{equation} \label{lem_exist.2}
\|\partial_xv_{n,m}\|_{L^1_TL^{\infty}_{xy}} \underset{n \to +\infty}{=} o(1)+O(\|\partial_x^{-1}\partial_yv_{n,m}\|_{L^{\infty}_TL^2_{xy}}) \, .
\end{equation} 
if the positive constant $A_s$ is chosen large enough in the definition of $T$ in \eqref{defT}.
\end{lemma}

\begin{proof} 
The proofs of both estimates follow the lines of the one of Lemma \ref{apriorilemma}. For the sake of brevity, we only give the proof of estimate \eqref{lem_exist.1}. Let us fix $\delta_0$ such that $0<\delta_0<s-s_{\alpha}$.

Since $v_{n,m}$ is a solution to \eqref{fKPdiffn}, we will use the refined Strichartz estimate \eqref{refstrichartz}. Hence we have (recall here that $0<T \le 1$)
\begin{eqnarray} 
\|v_{n,m}\|_{L^1_TL^{\infty}_x}&\lesssim& T^{\frac12}\|J^{s_{\alpha}-1+2\delta}_xv_{n,m}\|_{L^{\infty}_TL^2_{xy}}+T^{\frac12} \|J^{s_{\alpha}+\delta-1}_xD^{\delta}_yv_{n,m}\|_{L^{\infty}_TL^2_{xy}} \nonumber \\
&& +T^{\frac12} \int_0^T \|J^{s_{\alpha}-2+2\delta}_x\p_x((u_n+u_m)v_{n,m})\|_{L^2_{xy}}\,dt
\nonumber \\ &&+  
T^{\frac12}\int_0^T \|J^{s_{\alpha}+\delta-2}_xD^{\delta}_y\p_x((u_n+u_m)v_{n,m})\|_{L^{2}_{xy}}\,dt \nonumber \\
&=:&T^{\frac12} \big(D_1+D_2+D_3+D_4 \big) \, , \label{lem_exist.3}
\end{eqnarray}
for some small $0<\delta<\delta_0$ which will be determined during the proof.

By choosing $0<\delta<\frac{\delta_0}2$, it is clear from \eqref{existence.4} that 
\begin{equation} \label{lem_exist.4}
D_1  \underset{n \to +\infty}{=} o(n^{-1}) \, .
\end{equation}

To control $D_2$, we use Young's inequality with $p=\frac1{1-\delta}$ and $p'=\frac1{\delta}$ to obtain
\begin{equation} \label{refinedYoung}
\begin{split}
(1+|\xi|)^{s_{\alpha}+\delta-1} |\eta|^{\delta}&=(n^{\delta}(1+|\xi|)^{s_{\alpha}-1+2\delta} ) (n^{-1}\frac{|\eta|}{|\xi|})^{\delta}\\
&\lesssim n^{\frac{\delta}{1-\delta}}(1+|\xi|)^{\frac{s_{\alpha}-1+2\delta}{1-\delta}}+n^{-1}\frac{|\eta|}{|\xi|} \, .
\end{split}
\end{equation}
Then, it follows by using Plancherel's identity that 
\begin{displaymath}
D_2 \lesssim n^{\frac{\delta}{1-\delta}}\|J_x^{\frac{s_{\alpha}-1+2\delta}{1-\delta}}v_{n,m}\|_{L^{\infty}_TL^2_{xy}}+n^{-1}\|\partial_x^{-1}\partial_yv_{n,m}\|_{L^{\infty}_TL^2_{xy}} \, .
\end{displaymath}
To control the first term on the right-hand side, we use \eqref{existence.4} to get
\begin{displaymath}
n^{\frac{\delta}{1-\delta}}\|J_x^{\frac{s_{\alpha}-1+2\delta}{1-\delta}}v_{n,m}\|_{L^{\infty}_TL^2_{xy}} 
\lesssim n^{\frac{\delta}{1-\delta}}n^{-s+\frac{s_{\alpha}-1+2\delta}{1-\delta}} \lesssim n^{-1}n^{-s+s_{\alpha}+\delta_0}\underset{n \to +\infty}{=}o(n^{-1}) \, ,
\end{displaymath}
if $\delta$ is chosen small enough such that $(s_{\alpha}+2\delta)/(1-\delta)< s_{\alpha}+\delta_0$ .
Thus, we deduce that 
\begin{equation} \label{lem_exist.5}
D_2 \underset{n \to +\infty}{=} o(n^{-1}) +O(n^{-1}\|\partial_x^{-1}\partial_yv_{n,m}\|_{L^{\infty}_TL^2_{xy}}) \, .
\end{equation} 

Next we estimate $D_3$. Observe that
\begin{displaymath} 
\begin{split}
D_3 &\lesssim  \int_0^T \|(u_n-u_m)v_{n,m}\|_{L^2_{xy}}\,dt +\int_0^T \|D^{s_{\alpha}-1+2\delta}_x((u_n-u_m)v_{n,m})\|_{L^2_{xy}}dt \\
&= : D_{31}+D_{32}.
\end{split}
\end{displaymath}
Then, it follows by using \eqref{existence.3}
\begin{displaymath} 
D_{31} \lesssim \big( \|u_n\|_{L^1_TL^{\infty}_{xy}}+ \|u_m\|_{L^1_TL^{\infty}_{xy}}\big)
\|v_{n,m}\|_{L^{\infty}_TL^2_{xy}}\lesssim e^{cK}n^{-s} \, .
\end{displaymath}
To treat $D_{32}$ we use the Leibniz rule for fractional derivatives \eqref{Leibniz-1d} to deduce
\begin{displaymath}
\begin{split}
D_{32} 
&\lesssim  \Big(\|D^{s_{\alpha}-1+2\delta}_x u_n\|_{L^{\infty}_TL^2_{xy}}
+\|D^{s_{\alpha}-1+2\delta}_xu_m\|_{L^{\infty}_TL^2_{xy}}\Big) \|v_{n,m}\|_{L^1_TL^{\infty}_{xy}}\\
&\;\;\;\; +\Big(\|u_n\|_{L^1_TL^{\infty}_{xy}}+\|u_m\|_{L^1_TL^{\infty}_{xy}}\Big)\|D^{s_{\alpha}-1+2\delta}_x v_{n,m}\|_{L^{\infty}_TL^2_{xy}}
\end{split}
\end{displaymath}
Hence we conclude from \eqref{existence.4}  that
\begin{displaymath}
D_{32} \underset{n \to +\infty}{=} o(n^{-1}) +O(\|u_0\|_{X^s}\|v_{n,m}\|_{L^1_TL^{\infty}_{xy}})\, .
\end{displaymath}
Thus, we obtain that
\begin{equation}  \label{lem_exist.6}
D_3 \underset{n \to +\infty}{=} o(n^{-1}) +O(\|u_0\|_{X^s}\|v_{n,m}\|_{L^1_TL^{\infty}_{xy}})\, .
\end{equation}

Finally, we estimate $D_4$. Observe that 
\begin{displaymath} 
\begin{split}
D_4 &\lesssim \int_0^T \|D^{\delta}_y((u_n+u_m)v_{n,m})\|_{L^{2}_{xy}}\,dt +\int_0^T \|D^{s_{\alpha}+\delta-1}_xD^{\delta}_y((u_n+u_m)v_{n,m})\|_{L^{2}_{xy}}\,dt \\ 
&=: D_{41}+D_{42} \, .
\end{split}
\end{displaymath}
By using the fractional Leibniz rule \eqref{Leibniz-1d} in the $y$ variable and H\"older's inequality in the $x$-variable we have that 
\begin{equation*}
\begin{split} 
D_{41} &\lesssim  \big( \|D_y^{\delta}u_n\|_{L^{\infty}_TL^2_{xy}}+ \|D_y^{\delta}u_m\|_{L^{\infty}_TL^2_{xy}}\big)
\|v_{n,m}\|_{L^1_TL^{\infty}_{xy}}\\
&\hskip10pt +\big( \|u_n\|_{L^1_TL^{\infty}_{xy}}+ \|u_m\|_{L^1_TL^{\infty}_{xy}}\big)
\|D^{\delta}_yv_{n,m}\|_{L^{\infty}_TL^2_{xy}} \, .
\end{split}
\end{equation*}
Then, we use \eqref{existence.1} and argue as in \eqref{refinedYoung} to deduce that 
\begin{displaymath}
D_{41} \underset{n \to +\infty}{=} O(\|u_0\|_{X^s}\|v_{n,m}\|_{L^1_TL^{\infty}_{xy}}) +o(n^{-1}) + O(n^{-1}\|\partial_x^{-1}\partial_yv_{n,m}\|_{L^{\infty}_TL^2_{xy}})\, .
\end{displaymath}
Now, we have using the $2$ dimensional Leibniz rule \eqref{leibniz-2d} that 
\begin{displaymath}
\begin{split}
D_{42} &\lesssim \big( \|u_n\|_{L^1_TL^{\infty}_{xy}}+ \|u_m\|_{L^1_TL^{\infty}_{xy}}\big)
\|D_x^{s_{\alpha}-1+\delta}D^{\delta}_yv_{n,m}\|_{L^{\infty}_TL^2_{xy}} 
\\ &\quad +\big( \|D_x^{s_{\alpha}-1+\delta}D^{\delta}_yu_n\|_{L^{\infty}_TL^2_{xy}}+ \|D_x^{s_{\alpha}-1+\delta}D^{\delta}_yu_m\|_{L^{\infty}_TL^2_{xy}}\big)
\|v_{n,m}\|_{L^1_TL^{\infty}_{xy}}
\\ & \quad + \big( \|D_x^{s_{\alpha}-1+\delta}u_n\|_{L^1_TL^{\infty}_{xy}}+ \|D_x^{s_{\alpha}-1+\delta}u_m\|_{L^1_TL^{\infty}_{xy}}\big)
\|D^{\delta}_yv_{n,m}\|_{L^{\infty}_TL^2_{xy}} 
\\ &\quad +\big( \|D^{\delta}_yu_n\|_{L^{r_1}_TL^{p_1}_{xy}}+ \|D^{\delta}_yu_m\|_{L^{r_1}_TL^{p_1}_{xy}}\big)
\|D_x^{s_{\alpha}-1+\delta}v_{n,m}\|_{L^{s_1}_TL^{q_1}_{xy}} 
\\ &:= D_{421}+D_{422}+D_{423}+D_{424} \, ,
\end{split}
\end{displaymath}
where $1<r_1,s_1<\infty$ and $2<p_1,q_1<\infty$ are chosen as in Lemma \ref{interpolated} b). We easily see by using \eqref{refinedYoung} that
\begin{displaymath}
D_{421} \underset{n \to +\infty}{=}o(n^{-1}) + O(n^{-1}\|\partial_x^{-1}\partial_yv_{n,m}\|_{L^{\infty}_TL^2_{xy}})\, .
\end{displaymath}
Now, Young's inequality, Plancherel's identity and \eqref{existence.1} yield
\begin{displaymath}
D_{422} \lesssim \|u_0\|_{X^s}\|v_{n,m}\|_{L^1_TL^{\infty}_{xy}} \, .
\end{displaymath}
Moreover, we get from \eqref{kenig-(2.5)}, \eqref{existence.2} and arguing as in \eqref{refinedYoung} that
\begin{displaymath}
D_{423}  \underset{n \to +\infty}{=}o(n^{-1}) + O(n^{-1}\|\partial_x^{-1}\partial_yv_{n,m}\|_{L^{\infty}_TL^2_{xy}})\, .
\end{displaymath}
Finally, we deduce from \eqref{kenig-(2.6)}-\eqref{kenig-(2.7)}, \eqref{existence.1} and \eqref{existence.2} 
\begin{displaymath}
D_{424} \lesssim K^{1-\theta}\|u_0\|_{X^s}^{\theta}\|v_{n,m}\|_{L^1_TL^{\infty}_{xy}}^{\theta}
\|J_x^{s_{\alpha}-1+\delta_0}v_{n,m}\|_{L^{\infty}_TL^2_{xy}}^{1-\theta} \, ,
\end{displaymath}
so that 
\begin{displaymath}
D_{424} \underset{n \to +\infty}{=}o(n^{-1})+O(\|u_0\|_{X^s}\|v_{n,m}\|_{L^1_TL^{\infty}_{xy}} ) \, .
\end{displaymath}
as a consequence of Young's inequality and \eqref{existence.4}. Hence, we obtain gathering all those estimates that 
\begin{equation} \label{lem_exist.6a}
D_4 \underset{n \to +\infty}{=} o(n^{-1}) +O(\|u_0\|_{X^s}\|v_{n,m}\|_{L^1_TL^{\infty}_{xy}})+O(n^{-1}\|\partial_x^{-1}\partial_yv_{n,m}\|_{L^{\infty}_TL^2_{xy}}) \, .
\end{equation}

Therefore, we conclude gathering \eqref{lem_exist.3}--\eqref{lem_exist.6a} that
\begin{displaymath}
\|v_{n,m}\|_{L^1_TL^{\infty}_x}\underset{n \to +\infty}{=} o(n^{-1}) +O(T^{\frac12}\|u_0\|_{X^s}\|v_{n,m}\|_{L^1_TL^{\infty}_{xy}})+O(n^{-1}\|\partial_x^{-1}\partial_yv_{n,m}\|_{L^{\infty}_TL^2_{xy}}) \, ,
\end{displaymath}
which yields \eqref{lem_exist.1} if the constant $A_s$ is chosen large enough in \eqref{defT}.
\end{proof}

\begin{proof}[Proof of Proposition \ref{prop_exist}] We first prove \eqref{prop_exist.1}. 
We apply $\partial_x^{-1}\partial_y$ to \eqref{fKPdiffn} multiply by $\partial_x^{-1}\partial_yv_{n,m}$, integrate in space to deduce that
\begin{equation} \label{prop_exist.3}
\begin{split}
\frac12\frac{d}{dt}\|\p_x^{-1}\p_y v_{n,m}\|_{L^2_{xy}}^2&=
-\frac12 \int_{\mathbb R^2} \p_y((u_n+u_m)v_{n,m})\p_x^{-1}\p_y v_{n,m} \\
&=-\frac12\int_{\mathbb R^2} \p_y(u_n+u_m) v_{n,m} \p_x^{-1}\p_yv_{n,m} \\
&\hskip10pt+\frac14\int_{\mathbb R^2} \p_x(u_n+u_m) (\p_x^{-1}\p_yv_{n,m})^2\nonumber \\
&=: A_1+A_2 \, .
\end{split}
\end{equation}

On the one hand, H\"older's inequality implies that
\begin{equation} \label{prop_exist.4}
|A_2|\le \big(\|\partial_xu_n\|_{L^{\infty}_{xy}}+\|\partial_xu_m\|_{L^{\infty}_{xy}}\big) \|\p_x^{-1}\p_y v_{n,m}\|_{L^2_{xy}}^2  .
\end{equation}

On the other hand, we have still using H\"older's inequality
\begin{equation} \label{prop_exist.5}
|A_1|\le \big(\|\p_y u_n\|_{L^2_{xy}}+\|\p_y u_m\|_{L^2_{xy}}\big) \|v_{n,m}\|_{L^{\infty}_{xy}}\|\p_x^{-1}\p_y v_{n,m}\|_{L^2_{xy}}.
\end{equation}
Since $u_n$ is a solution to \eqref{fKP} we obtain after some integration by parts that
\begin{equation*}
\frac12\frac{d}{dt}\|\p_y u_n\|_{L^2_{xy}}^2=- \int_{\mathbb R^2} \p_y(u_n\p_x u_n)\p_y u_n=-\frac12\int_{\mathbb R^2} \p_xu_n(\p_y u_n)^2.
\end{equation*}
The Gronwall inequality yields then
\begin{equation*}
\|\p_y u_n(t)\|_{L^2_{xy}}\le \exp \Big(c\int_0^T\|\p_xu_n(t)\|_{L^{\infty}_{xy}} \,dt\Big)\|\p_yu_{0,n}\|_{L^2_{xy}}\le e^{cK}\|\p_yu_{0,n}\|_{L^2_{xy}} \, ,
\end{equation*}
where $K$ is defined in \eqref{existence.2}.
Thus, it follows from \eqref{BSreg.3} that
\begin{equation} \label{prop_exist.6}
\|\p_y u_n\|_{L^{\infty}_TL^2_{xy}}\lesssim  n \, .
\end{equation}
Therefore, we conclude from Gronwall's inequality combined with \eqref{prop_exist.3}--\eqref{prop_exist.6},\eqref{lem_exist.1} and  \eqref{BSreg.5} that 
\begin{displaymath} 
\|\partial_x^{-1}\partial_yv_{n,m}\|_{L^{\infty}_TL^2_{xy}} \le e^{cK} \big( \|\partial_x^{-1}\partial_y(u_{0,n}-u_{0,m})\|_{L^2}+o(1)  \big) \underset{n \to +\infty}{\longrightarrow} 0 \, ,
\end{displaymath}
which proves \eqref{prop_exist.1}.

\smallskip
Now, we prove \eqref{prop_exist.2}. We apply $J_x^s$ to \eqref{fKPdiffn} multiply by $J^s_xv_{n,m}$, integrate in space to deduce that
\begin{equation}\label{prop_exist.7}
\begin{split}
-\frac12\frac{d}{dt}\|J_x^s v_{n,m}\|_{L^2_{xy}}^2&=
\int_{\mathbb R^2} J^s_x(\partial_x(u_n+u_m)v_{n,m})J_x^sv_{n,m} \\
&\hskip10pt+\int_{\mathbb R^2} J^s_x((u_n+u_m)\partial_xv_{n,m})J_x^sv_{n,m} \\
&=: B_1+B_2 \, . 
\end{split}
\end{equation}

Firstly, we obtain after integrating by parts that
\begin{displaymath} 
B_2 =\int_{\mathbb R^2} [J_x^sv_{n,m}, u_n+u_m]\partial_xv_{n,m} \, J^s_xv_{n,m}+\frac12\int_{\mathbb R^2}\partial_x(u_n+u_m) (J_x^sv_{n,m})^2 \, .
\end{displaymath}
This implies 
\begin{equation} \label{prop_exist.8}
|B_2| \lesssim \|\partial_x(u_n+u_m)\|_{L^{\infty}_{xy}}\|J_x^sv_{n,m}\|_{L^2_{xy}}^2
+\|J^s_x(u_n+u_m)\|_{L^2_{xy}} \|\partial_xv_{n,m}\|_{L^{\infty}_{xy}}\|J_x^sv_{n,m}\|_{L^2_{xy}} \, .
\end{equation}
due to estimate \eqref{Kato-Ponce.1} and H\"older's inequality. 

Secondly, H\"older's inequality yields
\begin{equation} \label{prop_exist.9}
|B_1| \lesssim \big(\|\partial_x(u_n+u_m)\|_{L^{\infty}_{xy}}\|v_{n,m}\|_{L^2_{xy}}+\|D_x^s(\partial_x(u_n+u_m)v_{n,m})\|_{L^2_{xy}} \big)\|J_x^sv_{n,m}\|_{L^2_{xy}} \, .
\end{equation}
Observe that $D^s_x=D^{s-1}_x\mathcal{H}_x\partial_x$, where $\mathcal{H}_x$ denotes the Hilbert transform in the $x$ variable is a bounded operator in $L^2_x$. Then, we can estimate the last term on the right-hand side of \eqref{prop_exist.9} as 
\begin{eqnarray} \label{prop_exist.10}
\|D_x^s(\partial_x(u_n+u_m)v_{n,m})\|_{L^2_{xy}} &\le& \|D_x^{s-1}(\partial_x(u_n+u_m)\partial_xv_{n,m})\|_{L^2_{xy}}
\nonumber \\ &&+\|D_x^{s-1}(\partial_x^2(u_n+u_m)v_{n,m})\|_{L^2_{xy}}  \\ 
&=:& B_{11}+B_{12} \, .  \nonumber
\end{eqnarray}
By using the fractional Leibniz rule \eqref{Leibniz-1d} in the $x$-variable (recall here that $0<s-1<1$), we get that
\begin{equation} \label{prop_exist.11}
|B_{11}| \lesssim \|J^s_x(u_n+u_m)\|_{L^2_{xy}} \|\partial_xv_{n,m}\|_{L^{\infty}_{xy}}+\|J^s_xv_{n,m}\|_{L^2_{xy}} \|\partial_x(u_n+u_m)\|_{L^{\infty}_{xy}}
\end{equation}
and 
\begin{equation} \label{prop_exist.12}
|B_{12}| \lesssim \|J^{s+1}_x(u_n+u_m)\|_{L^2_{xy}} \|v_{n,m}\|_{L^{\infty}_{xy}}+\|D^{s-1}_xv_{n,m}\|_{L^2_{xy}} \|\partial_x^2(u_n+u_m)\|_{L^{\infty}_{xy}} \, .
\end{equation}

To control $\|J^{s+1}_xu_n\|_{L^2_{xy}}$,  observe from \eqref{enest.1} that
\begin{equation*}
\frac{d}{dt}\|J_x^{s+1} u_n\|_{L^2_{xy}}^2 \lesssim \|\partial_xu_n\|_{L^{\infty}_{xy}} \|J^{s+1}_xu_n\|_{L^2_{xy}}^2 \, ,
\end{equation*}
since $u_n$ is a solution to \eqref{fKP}. The Gronwall inequality yields then
\begin{equation*}
\begin{split}
\|J_x^{s+1} u_n(t)\|_{L^2_{xy}}&\le \exp \Big(c\int_0^T\|\p_xu_n(t)\|_{L^{\infty}_{xy}} \,dt\Big)\|J_x^{s+1}u_{0,n}\|_{L^2_{xy}}\\ &\le e^{cK}\|J_x^{s+1}u_{0,n}\|_{L^2_{xy}} \, ,
\end{split}
\end{equation*}
where $K$ is defined in \eqref{existence.2}.
Thus, it follows from \eqref{BSreg.1} that
\begin{equation} \label{prop_exist.13}
\|J_x^{s+1} u_n\|_{L^{\infty}_TL^2_{xy}}\lesssim  n \, .
\end{equation}

To control $\|\partial_x^2u_n\|_{L^{\infty}_{xy}}$, we combine \eqref{existence.1}, \eqref{existence.200},\eqref{prop_exist.13} and deduce that 
\begin{equation} \label{prop_exist.14}
\|\partial_x^2u_n\|_{L^1_TL^{\infty}_{xy}} \lesssim  n \, .
\end{equation}

Therefore, we conclude combining Lemma \ref{lem_exist}, \eqref{prop_exist.7}--\eqref{prop_exist.14}, \eqref{BSreg.4}  and using Gronwall's inequality that
\begin{displaymath} 
\|J_x^sv_{n,m}\|_{L^{\infty}_TL^2_{xy}} \le e^{cK} \big( \|J_x^s(u_{0,n}-u_{0,m})\|_{L^2}+o(1)  \big) \underset{n \to +\infty}{\longrightarrow} 0 \, .
\end{displaymath}
This finishes the proof of \eqref{prop_exist.2}.
\end{proof}

We deduce from Proposition \ref{prop_exist} and Lemma \ref{lem_exist} that $\{u_n\}$ is a Cauchy sequence in $L^{\infty}([0,T] : X^s(\mathbb R^2)) \cap L^1(0,T : W_x^{1,\infty}(\mathbb R^2))$.
Hence, there exists $u \in C([0,T] : X^s(\mathbb R^2)) \cap L^1(0,T : W_x^{1,\infty}(\mathbb R^2))$ such that 
\begin{equation} \label{existence.15}
\|u_n-u\|_{L^{\infty}_TX^s}+ \|u_n-u\|_{L^1_TW^{1,\infty}_{x}}\underset{n \to +\infty}{\longrightarrow} 0 \, .
\end{equation}
This allows to pass to the limit in \eqref{fKPun} and deduce that $u$ is a solution to the Cauchy problem \eqref{fKP}. 

\subsection{Continuity of the flow map data-solution} 
Once again, assume that $0<\alpha<2$ and $2-\frac{\alpha}2< s <2$. Fix $u_0 \in X^s(\mathbb R^2)$. By the existence and uniqueness part, we know that there exist a positive time $T=T(\|u_0\|_{H^s})$ and a unique solution $u \in C([0,T]:X^s(\mathbb R^2)) \cap L^1((0,T) : W_x^{1,+\infty}(\mathbb R^2))$ to \eqref{fKP}. Since $T$ is a nonincreasing function of its argument, for any $0<T'<T$, there exists a small ball $B_{\tilde{\delta}}(u_0)$ of $X^s$ centered in $u_0$ and of radius $\tilde{\delta}>0$, \textit{i.e.}
 \begin{displaymath}
 B_{\tilde{\delta}}(u_0)=\big\{v_0 \in X^s(\mathbb R) \ : \ \|v_0-u_0\|_{X^s} < \tilde{\delta} \big\},
 \end{displaymath} 
 such that for each $v_0 \in B_{\tilde{\delta}}(u_0)$, the solution $v$ to \eqref{fKP} emanating from $v_0$ is defined at least on the time interval $[0,T']$.
 
Let $\epsilon>0$ be given. It suffices to prove that there exists $\delta=\delta(\theta)$ with $0<\delta<\tilde{\delta}$ such that for any initial data $v_0 \in X^s(\mathbb R^2)$ with $\|u_0-v_0\|_{X^s} < \delta$, the solution $v \in C([0,T'];X^s(\mathbb R)^2)$  emanating from $v_0$ satisfies 
\begin{equation} \label{continuity.1} 
\|u-v\|_{L^{\infty}_{T'}X^s_{xy}} < \epsilon \, .
\end{equation}

For any $n \in \mathbb N$, $n \ge 1$, we regularize the initial data $u_0$ and $v_0$ by defining $u_{0,n}=P_{\le n} u_0$ and $v_{0,n}=P_{\le n}v_0$ as in the previous subsection and consider the associated smooth solutions $u_n, \ v_n \in C([0,T'];X^{\infty}(\mathbb R^2))$. Then it follows from the triangle inequality that 
\begin{equation} \label{continuity.2} 
\|u-v\|_{L^{\infty}_{T'}X^s} \le \|u-u_n\|_{L^{\infty}_{T'}X^s}+\|u_n-v_n\|_{L^{\infty}_{T'}X^s}
+\|v-v_n\|_{L^{\infty}_{T'}H^s_x} \ .
\end{equation}
On the one hand, according to \eqref{existence.15}, we can choose $n_0$ large enough so that 
\begin{equation} \label{continuity.3} 
\|u-u_{n_0}\|_{L^{\infty}_{T'}X^s}+\|v-v_{n_0}\|_{L^{\infty}_{T'}X^s} < \frac{2}3 \epsilon \, .
\end{equation}
On the other hand, we get from \eqref{BSreg.1} and \eqref{BSreg.2} that 
\begin{displaymath} 
\|u_{0,n_0}-v_{0,n_0}\|_{X^5} \lesssim n_0^{5-s}\|u_0-v_0\|_{X^s} \lesssim n_0^{5-s}\delta \, .
\end{displaymath}
Therefore, by using the continuity of the flow map for initial data in $X^5(\mathbb R^2)$ (c.f. Lemma \ref{l1}), we can choose $\delta>0$ small enough such that 
\begin{equation} \label{continuity.4} 
\|u_{n_0}-v_{n_0}\|_{L^{\infty}_{T'}H^s_x} < \frac13 \epsilon \, .
\end{equation}
Estimate \eqref{continuity.1} is concluded gathering \eqref{continuity.2} and \eqref{continuity.4}.

This concludes the proof of Theorem \ref{maintheo}.

\section{Final remarks}

As was proven in Section 3 the fKP-I equation is quasilinear when $0<\alpha\leq 2,$ while (see \cite{ST}) it is semilinear when $\alpha=4.$ It would be interesting to find the critical value of $\alpha$ that separates the two regimes.

We also shall mention that for the KP-I equation it was shown in \cite{KoTz2} that the data-solution map is not uniformly continuous in bounded
sets of the natural energy space. We do not know whether  this is true for the fKP-I equation.

We were mainly concerned in this paper to the {\it local} well-posedness of  the Cauchy problem. Here are some more qualitative issues that we plan to study in subsequent works.

It has been proven in \cite{L1, LW} (see also \cite{KS} for numerical simulations) that for the generalized KP-I equation

\begin{equation}\label{GKP}
u_t+u^pu_x+u_{xxx}-\partial_x^{-1}u_{yy}=0
\end{equation}
a finite time blow-up may occur in the $L^2$ supercritical case $p>\frac{4}{3}.$ This blow-up is the mechanism of the instability of the ground  states.

One expects a similar phenomenum for the $L^2$- supercritical fKP-I equation with $\frac{4}{5}<\alpha <\frac{4}{3}$  and also for the $L^2$ critical case $\alpha =\frac{4}{3}.$ 

In particular a finite
 time blow-up is expected for the KP-I version of the BO or ILW equations. 

On the other hand, one expects the orbital stability of the set of ground states of the fKP-I equation in the $L^2$ subcritical case $\frac{4}{3}<\alpha<2,$ as it is the case for the usual KP-I equation ($\alpha =2$), see \cite{deBS3}.

 A finite time blow-up (of a different nature though) is also expected in the energy supercritical range   $0<\alpha<\frac{4}{5}$ both for the fKP-I and fKP-II equations.

 Another interesting issue is that  of the transverse stability of the line solitary wave. Of special interest is the transverse stability of the (explicit) soliton of the BO and ILW equations with respect to their KP-II version. Recall that the KdV soliton is transversally stable with respect to the KP-II equation (\cite{RTz,RTz2}) and a similar result is expected for the BO-KP-II and ILW-KP -II equations.
\appendix
\section{Proof of Lemma \ref{interpolated}}
The argument of proof is analogous to the one given by Kenig in \cite{K}.  

\begin{proof}[Proof of inequality \eqref{kenig-(2.5)}]
To prove  \eqref{kenig-(2.5)} we use a Littlewood-Paley decomposition of $u$ in the $\xi$ variable.  Let $\varphi \in C^{\infty}_0(\frac12<|\xi|<2)$ and $\chi\in C_0^{\infty}(|\xi|<2)$ such that 
$\underset{k=1}{\overset{\infty}{\sum}} \varphi(2^{-k}\xi) +\chi(\xi)=1.$

For $\lg=2^k$, $k\ge 1$, define $u_{\lg}= Q_k u$, where $\widehat{Q_ku}(\xi,\eta)= \varphi(2^{-k}\xi)\widehat{u}(\xi,\eta)$,
$u_0= Q_0 u$, and $\widehat{Q_0u}(\xi,\eta)= \chi(\xi)\widehat{u}(\xi,\eta)$.  Let $u_1= \underset{k\ge1}{\sum} Q_ku$. Thus $u=u_0+u_1$.

First we estimate $D^{s_{\alpha}-1+\delta}_xu_0$. Observing that $\widehat{f}(\xi)=|\xi|^{s_{\alpha}-1+\delta}\chi(\xi)$ implies that $f\in L^1(\R)$. Then
\begin{equation}\label{A1}
\|D_x^{s_{\alpha}-1+\delta}u_0\|_{L^{\infty}_{xy}}\le \|u\|_{L^{\infty}_{xy}}.
\end{equation}

To estimate $D^{s_{\alpha}-1+\delta}_xu_1= \underset{k\ge1}{\sum} D^{s_{\alpha}-1+\delta}_xQ_ku$, observe that
\begin{equation*}
\begin{split}
\big(D^{s_{\alpha}-1+\delta}_xQ_ku\big)^{\wedge}(\xi,\eta)&= |\xi|^{s_{\alpha}-1+\delta}\,\varphi(2^{-k}\xi)\widehat{u}(\xi,\eta)=\frac{ |\xi|^{s_{\alpha}-1+\delta}}{\xi}\xi\,\varphi(2^{-k}\xi)
\widehat{u}(\xi,\eta)\\
&= c2^{k(s_{\alpha}-2+\delta)}\frac{|2^{-k}\xi|^{s_{\alpha}-1+\delta}}{(2^{-k}\xi)} \,\varphi(2^{-k}\xi)\,\widehat{\partial_xu}(\xi,\eta).
\end{split}
\end{equation*}
and that if $\;\widehat{g}(\xi)=\dfrac{|\xi|^{s_{\alpha}-1+\delta}}{\xi} \,\varphi(\xi)$, then $g\in L^1(\R)$. Thus
\begin{equation}\label{A2}
\|D^{s_{\alpha}-1+\delta}_xQ_ku\|_{L^{\infty}_{xy}}\le  2^{k(s_{\alpha}-2+\delta)}\|\partial_xu\|_{L^{\infty}_{xy}}.
\end{equation}

We conclude the proof of  inequality \eqref{kenig-(2.5)} gathering \eqref{A1} and \eqref{A2} and observing that $s_{\alpha}-2+\delta=-\frac{\alpha}4+\delta<0$, since $0<\delta<\frac{\alpha}4$.
\end{proof}

\begin{proof}[Proof of inequality \eqref{kenig-(2.6)}]

To show \eqref{kenig-(2.6)} we use the Littlewood-Paley decomposition in the $x$ variable as above. Thus
\begin{equation*}
\partial_xu=\underset{k\ge0}{\sum} Q_k(\partial_xu) \text{\hskip15pt and \hskip10pt} 
D^{s_{\alpha}-1+\delta}_x \partial_xu= \underset{k\ge0}{\sum} D^{s_{\alpha}-1+\delta}_x Q_k(\partial_xu).
\end{equation*}

For $k\ge1$
\begin{equation*}
D^{s_{\alpha}-1+\delta}_x Q_k= 2^{k(s_{\alpha}-1+\delta)} \widetilde{Q}_k
\end{equation*}
where $\widehat{ \widetilde{Q}_k f}(\xi)= \widetilde{\varphi}(2^{-k}\xi) \widehat{f}(\xi)$ and $\widetilde{\varphi}\;$ has similar
properties as $\varphi$, and $D^{s_{\alpha}-1+\delta}_x Q_0= \widetilde{Q}_0$, with 
$\widehat{ \widetilde{Q}_0 f}(\xi)= \widetilde{\chi}(\xi)\widehat{f}(\xi)$, and if $\;\widehat{g}=\widetilde{\chi}\;$ then $g\in L^1(\R)$. Thus
we have
\begin{equation*}
\|D^{s_{\alpha}-1+\delta}_x\partial_x u\|_{L^{s_1}_TL^{q_1}_{xy}}\le \underset{k\ge0}{\sum} 2^{k(s_{\alpha}-1+\delta)}
\|\widetilde{Q}_k \partial_x u\|_{L^{s_1}_TL^{q_1}_{xy}}.
\end{equation*}

Consider now $k=0$, then we have
\begin{equation}
\|\widetilde{Q}_0 \partial_x u\|_{L^{s_1}_TL^{q_1}_{xy}}\le \|\widetilde{Q}_0 \partial_x u\|_{L^1_TL^{\infty}_{xy}}^{\theta}
\|\widetilde{Q}_0 \partial_x u\|_{L^{\infty}_TL^2_{xy}}^{1-\theta} \, ,
\end{equation}
where
\begin{equation}\label{conjugates}
\dfrac{1}{s_1}=\theta, \quad \dfrac{1}{q_1}=\dfrac{1-\theta}{2} \quad \text{and} \quad 0\le \theta\le 1 \, .
\end{equation}

For $k\ge 1$, using \eqref{conjugates} it follows that
\begin{equation*}
\begin{split}
\|\widetilde{Q}_k \partial_x u\|_{L^{s_1}_TL^{q_1}_{xy}}&\le 
 \|\widetilde{Q}_k \partial_x u\|_{L^1_TL^{\infty}_{xy}}^{\theta}
\|\widetilde{Q}_k \partial_x u\|_{L^{\infty}_TL^2_{xy}}^{1-\theta}\\
&\lesssim  \|\widetilde{Q}_k \partial_x u\|_{L^1_TL^{\infty}_{xy}}^{\theta} \; 2^{-k(s_{\alpha}-1+\delta_0)(1-\theta)}  
\|\widetilde{\widetilde{Q}}_k D^{s_{\alpha}-1+\delta_0}_x \partial_x u\|_{L^{\infty}_TL^2_{xy}}^{1-\theta}\\
&\lesssim  \|\widetilde{Q}_k \partial_x u\|_{L^1_TL^{\infty}_{xy}}^{\theta} \; 2^{-k(s_{\alpha}-1+\delta_0)(1-\theta)}  
\|D^{s_{\alpha}+\delta_0}_x u\|_{L^{\infty}_TL^2_{xy}}^{1-\theta} \, ,
\end{split}
\end{equation*}
where $\widetilde{\widetilde{Q}}_k$ is associated to $\widetilde{\widetilde{\varphi}}(2^{-k}\xi)$ with $\widetilde{\widetilde{\varphi}}\;$
having similar properties as $\widetilde{\varphi}$.

Hence \eqref{kenig-(2.6)} follows, as long as
\begin{equation*}
s_{\alpha}-1+\delta < \Big(s_{\alpha}-1+\delta_0\Big) (1-\theta)
\end{equation*}
or choosing $\delta$ small enough, as long as,
\begin{equation} \label{condtheta}
s_{\alpha}-1< \Big(s_{\alpha}-1+\delta_0\Big) (1-\theta) \quad \Leftrightarrow \quad 
0<\theta< 1-\frac{s_{\alpha}-1}{s_{\alpha}-1+\delta_0} \, .
 \end{equation}
\end{proof}

\begin{proof}[Proof of inequality \eqref{kenig-(2.7)}]

To establish \eqref{kenig-(2.7)} we use a Littlewood-Paley decomposition in the $y$ variable. Proceeding as before we define
\begin{equation*}
D^{\delta}_yQ_k= 2^{k\delta}\widetilde{Q}_k \text{\hskip15pt where \hskip10pt} \widehat{\widetilde{Q}_k f} (\eta)= \widetilde{\varphi}(2^{-k}\eta) \widehat{f}(\eta),
\end{equation*}
and 
\begin{equation*}
D^{\delta}_y Q_0= \widetilde{Q}_0 \text{\hskip15pt where \hskip10pt} \widehat{\widetilde{Q}_0f} (\eta)= \widetilde{\chi}(2^{-k}\eta) \widehat{f}(\eta).
\end{equation*}
As before if $\widehat{g}(\eta)=\widetilde{\chi}(\eta)$, then $g\in L^1(\R)$.

First, observe that
\begin{equation}\label{(2.7)-step1}
\|D^{\delta}_y u\|_{L^{r_1}_TL^{p_1}_{xy}}\le \underset{k\ge0}{\sum} \; 2^{k\delta}
\|\widetilde{Q}_k u\|_{L^{r_1}_TL^{p_1}_{xy}}.
\end{equation}
To bound the right hand side of \eqref{(2.7)-step1}, we deduce interpolating that
\begin{equation*}
\|\widetilde{Q}_k u\|_{L^{r_1}_TL^{p_1}_{xy}}\le \|\widetilde{Q}_k u\|_{L^1_TL^{\infty}_{xy}}^{1-\theta} \|\widetilde{Q}_k u\|_{L^{\infty}_TL^{2}_{xy}}^{\theta} \, ,
\end{equation*}
where 
\begin{equation} \label{conjugates2}
\dfrac{1}{r_1}=1-\theta, \quad \dfrac{1}{p_1}=\dfrac{\theta}{2} \quad \text{and} \quad 0\le \theta\le 1 \, .
\end{equation}

Therefore, to conclude the proof of \eqref{kenig-(2.7)}, it is enough to choose  $\delta<\frac\theta2$, where we recall that 
$\theta$ must satisfy \eqref{condtheta}.
\end{proof}

\section{Proof of Lemma \ref{anisSobo}}Let $s>4$.
Recall from the definition of $X^s$ in \eqref{Xnorm} that
\begin{equation} \label{Xs-Fourier}
\|u\|_{X^s} \sim \left( \int_{\mathbb R^2}\Big( (1+\xi^2)^s+\Big( \frac{\eta}{\xi}\Big)^2\Big) |\widehat{u}(\xi,\eta)|^2 \, d\xi d\eta \right)^{\frac12} \, .
\end{equation}

Now, we get by using the inverse Fourier formula and the Cauchy-Schwarz inequality that 
\begin{displaymath}
\|\partial_x u\|_{L^{\infty}_{xy}} \lesssim \int_{\mathbb R^2} |\xi| |\widehat{u}(\xi,\eta)| \, d\xi d\eta 
 \lesssim \left( \int_{\mathbb R^2} (1+\xi)^{1+\delta}(1+\eta)^{1+\delta} |\xi|^2|\widehat{u}(\xi,\eta)|^2 \, d\xi d\eta \right)^{\frac12} \, ,
\end{displaymath}
where $\delta$ denotes any small positive number. Moreover, Young's inequality yield
\begin{displaymath}
\begin{split}
(1+\xi)^{1+\delta}(1+\eta)^{1+\delta} |\xi|^2 &\lesssim \big(1+|\xi|^{1+\delta}+|\eta|^{1+\delta}+|\eta \xi|^{1+\delta} \big)|\xi|^2  \\ &\lesssim 1+\Big( \frac{\eta}{\xi}\Big)^2+|\xi|^{\frac{8+2\delta}{1-\delta}} \, .
\end{split}
\end{displaymath}
Therefore, we conclude the proof of inequality \eqref{anisSobo.1} by choosing $\delta>0$ such that $\frac{8+2\delta}{1-\delta}<2s$.

\begin{merci}
The Authors were partially  supported by the Brazilian-French program in mathematics and the MathAmSud program. J.-C. S. acknowledges support from the project ANR-GEODISP of the Agence Nationale de la Recherche. F.L and D.P. were partially supported by CNPq and FAPERJ/Brazil. The authors would like to thank H. Koch for pointing out reference \cite{Ha}.
\end{merci}
\bibliographystyle{amsplain}

\end{document}